\documentclass[6pt]{article}
\usepackage[fleqn]{amsmath}
\usepackage{amssymb, amsthm, latexsym, amsfonts, epsfig,  graphicx,epstopdf,subcaption,caption,mathrsfs}
\usepackage{float}
\usepackage{ulem}
\usepackage{enumitem}
 \usepackage{color}
\numberwithin{equation}{section}

\def\kasten{$~~\mbox{\hfil\vrule height6pt width5pt depth-1pt}$ }

\usepackage{geometry}
\newtheorem{Theorem}{Theorem}[section]
\newtheorem{Corollary}[Theorem]{Corollary}
\newtheorem{Definition}[Theorem]{Definition}
\newtheorem{Proposition}[Theorem]{Proposition}
\newtheorem{Lemma}[Theorem]{Lemma}
\newtheorem{Example}[Theorem]{Example}
\newtheorem{Remark}[Theorem]{Remark}
\newtheorem{Claim}[Theorem]{Claim}
\begin{document}

\def\kasten{$~~\mbox{\hfil\vrule height6pt width5pt depth-1pt}$ }
\makeatletter\def\theequation{\arabic{section}.\arabic{equation}}
\makeatother

\newtheorem{theorem}{Theorem}
\centerline{\large \bf A Wong-Zakai Approximation for Random Slow Manifolds}
\centerline{\large \bf \qquad\qquad\quad with Application to Parameter Estimation
\footnote{This work was partly supported
by the
 NSFC grants 11301197,   11301403,   11371367 and 11271290.}\hspace{2mm}
\vspace{1mm}\vspace{1mm}\\ }
\bigskip
\centerline{\bf Ziying He$^{a,b,}\footnote{ziyinghe@hust.edu.cn}$,
Xinyong Zhang$^{c,}\footnote{zhangxinyong12@mails.tsinghua.edu.cn}$,
Tao Jiang$^{d,}\footnote{tjiang1985@gmail.com}$ and 
Xianming Liu$^{b,}\footnote{xmliu@hust.edu.cn}$}
\smallskip
\centerline{${}^a$ Center for Mathematical Sciences,} 
 \centerline{${}^b$School of Mathematics and Statistics,}
 \centerline{Hubei Key Laboratory of Engineering Modeling and Scientific Computing,} 
\centerline{Huazhong University of Sciences and Technology, Wuhan 430074,  China} 
\centerline{${}^c$ Department of Mathematical Sciences} \centerline{Tsinghua University, Beijing 100084,  China} 
\centerline{${}^d$ Collaborative Innovation Center of China Pilot Reform Exploration and Assessment,} \centerline{ Hubei Sub-Center,}
\centerline{Hubei University of Economics, Wuhan, 430205,  China} 

\begin{abstract}
We study a Wong-Zakai approximation for the random slow manifold of a slow-fast stochastic dynamical system. We first deduce the existence of the random slow manifold about an approximation system driven by  an integrated Ornstein-Uhlenbeck (O-U) process.  Then we compute the  slow manifold of the approximation system, in order to gain insights of the long time dynamics of the original stochastic system. By restricting this approximation system to its slow manifold, we  thus get a reduced slow random system. This reduced slow random system is used to accurately estimate a system parameter of the original system. An example is presented to illustrate this approximation.
\end{abstract}

\medskip\par\noindent
{\bf Key Words and Phrases}: Random slow manifold, multi-scale dynamics,  Wong-Zakai approximation, integrated O-U processes, parameter estimate\\

{\footnotesize \textbf{2010 Mathematics Subject Classification}: Primary: 37L55, 35R60; Secondary: 60H15, 58J65.}

\section{Introduction}
An Ornstein-Uhlenbeck (O-U) process is introduced historically to depict the velocity of the particle in Brownian motion. Its integration, the integrated O-U process, is regarded as the displacement of the particle \cite{G. E. Uhlenbeck and L. S. Ornstein}. The O-U process has a frequency-dependent power spectral density. Due to this characteristic behavior of its power spectral density, it is also called a colored noise to be distinguished from white noise \cite{R. Kazakevicius and J. Ruseckas}. When the correlation of collisions between the Brownian particle and the surrounding liquid molecules leads to a situation where the finite correlation time becomes important,  the system driven by colored noise instead of white noise deserves investigation \cite{R. Kazakevicius and J. Ruseckas}.  Numerous realistic models \cite{E. Bibbona and G. Panfilo and P. Tavella, W. Horsthemke and R. Lefever, R. Kazakevicius and J. Ruseckas},  analytical works \cite{T. Blass and L. A. Romero,  R. Hintze and I. Pavlyukevich} and numerical simulations  \cite{M. Kamrani, Q. Du and T. Zhang} concern the system driven by colored noise. When the correlation time tends to vanish, the colored noise will approach the white noise. According to the  Wong-Zakai theorem \cite{wongzakai, wongzakai2}, the system driven by colored noise (or integrated O-U process) will converge to the system driven by white noise (or Brownian motion). Inheriting the benefits of the general Wong-Zakai approximation, the system driven by integrated O-U process is a random differential equation (RDE). The RDE can be regarded as the ordinary differential equation (ODE) pathwisely. And it can be represented by deterministic Riemann integral which is more robust to approximate than stochastic integral in view of their definition. Thus it is easier to simulate a RDE than a stochastic differential equation (SDE) \cite{F. Cucker, P. E. Kloeden and A. Jentzen}. But differing from the piecewise linear approximation to Brownian motion \cite{brzf}, the integrated O-U process has a continuous derivative. This analytic property makes the system easier to analyze to some extent \cite{P. Acquistapace}. Due to these  characteristics, it is meaningful to study both the dynamical behavior of system driven by integrated O-U process and the approximation property when the correlation time tends to infinity. In fact, Wong-Zakai theorem is extended to other situations \cite{brz,brzf,igy1, igy2,mar,igy, kon}. The dimension of the state space has grown from one to finite, and then infinity. That means the state space can be a general Hilbert space. The driving process is extended from Brownian motion to semimartingales. Various modes of convergence are considered, such as convergence in the mean square, in probability, and almost surely. The rate of the convergence has also been examined  \cite{brz,igy}, which improves the accuracy of the approximation to a higher level.

Many dynamical systems involve the interplay of two time scales. For example, the Lorenz-Krishnamurthy model for inertia-gravity waves depicting the circulation of  atmosphere and ocean, the FitzHugh-Nagumo system which is a simplification of the Hodgkin-Huxley model for an electric potential of a nerve axon, the van der Pol  oscillator for a vacuum tube triode circuit, the settling of inertial particles under uncertainty, and the stiff stochastic chemical systems \cite{N. Berglund and B. Gentz, Dimitris A. Goussis, Xiaoying Han, ChristianK, E.N. Lorenz, Ren jian 2}. The dimension of a two-scale system can be reduced to that of the slow variables by a slow manifold with exponential tracking property if it exists. The slow manifold is considered in Fenichel's theorem in singular perturbation theory \cite{ChristianK} (corresponding to the adiabatic manifold \cite{N. Berglund and B. Gentz}). This concept stems from Leith (1980) for weather forecast, then is explored by Lorenz through Lorenz-Krishnamurthy model \cite{C. E. Leith, E.N. Lorenz2, E.N. Lorenz}. The slow manifold of slow-fast system is a special invariant manifold which is studied extensively \cite{N. Berglund and B. Gentz, J. Duan W. Wang, Dimitris A. Goussis, Xiaoying Han, ChristianK, E.N. Lorenz, Ren jian 2}. Solutions on it evolve relatively slow compared to the fast variables. If the fast variables decay with an exponential velocity, then they can be eliminated by confining trajectories to the slow manifold. The random slow manifold of a two-scale stochastic partial differential equation (SPDE) driven by Brownian motion in Hilbert space is considered in \cite{HongboFuSlowmf}. There exists a Lipschitz random slow manifold with exponential tracking property in a two-scale SPDE under suitable conditions \cite{HongboFuSlowmf}. Since the random invariant manifold has a Wong-Zakai approximation with Brownian motion replaced by integrated O-U process \cite{TaoJiang} which bears many benefits mentioned before. To conduct a research about the Wong-Zakai approximation for the random slow manifold of a two-scale SPDE is meaningful and feasible. Furthermore, the reduced system by confining trajectories to the random slow manifold   captures some quantitative properties about the original system. Moreover, this reduced, slow system
also provides an accurate estimate on the original system parameter, which  reduces the amount of information needed before making an estimation and lows the cost to simulate \cite{Ren jian}. The computational method of parameter estimate will be further simplified by replacing the original random slow manifold by its Wong-Zakai approximation. This is due to the simulating robustness of RDEs compared to SDEs, which is mentioned in the previous paragraph.

In this paper, we consider the random slow manifold and its Wong-Zakai approximation for a slow-fast system of the SPDE. The settings and main results are  in Section \ref{setresult}.
Section \ref{convert} is about converting the original system and the Wong-Zakai approximation system into comparable random partial differential equation (RPDE) systems.
In Section \ref{pfexwsm}, we prove that the  RPDE (\ref{eq:newRDE}) driven by the colored noise exists a Lipschitz random slow manifold with exponential tracking property. Furthermore we prove that this random slow manifold of the RPDE with colored noise  approximates that of the SPDE with white noise. In addition, the random slow manifold of the RPDE with colored noise can exponentially track all orbits of the SPDE with white noise. This permits us to project the SPDE with white noise to the random slow manifold of the RPDE with colored noise, and get a lower dimensional deterministic system pathwisely. In Section \ref{sec:estimate}, we show that this Wong-Zakai type reduced system can quantify the unknown system parameter with a high accuracy. We illustrate the previous procedure through a simple example.

\section{Settings and main results}\label{setresult}
\subsection{Settings}
It is already known that Brownian motion $B_t$ can be approximated uniformly by O-U processes almost surely on a finite interval   \cite{P. Acquistapace}. This fact is shown with a few pictures by different samples. Figure \ref{path}   indicates that the sharp points of Brownian motion are smoothed by an integrated O-U process (see $\Phi^{\mu}_{t}(\omega)$ below). This is coincide with the fact that the path of Brownian motion is nowhere differentiable, while that of integrated O-U process is continuously differentiable.
\begin{figure}[H]
\centering
\includegraphics[width=2in]{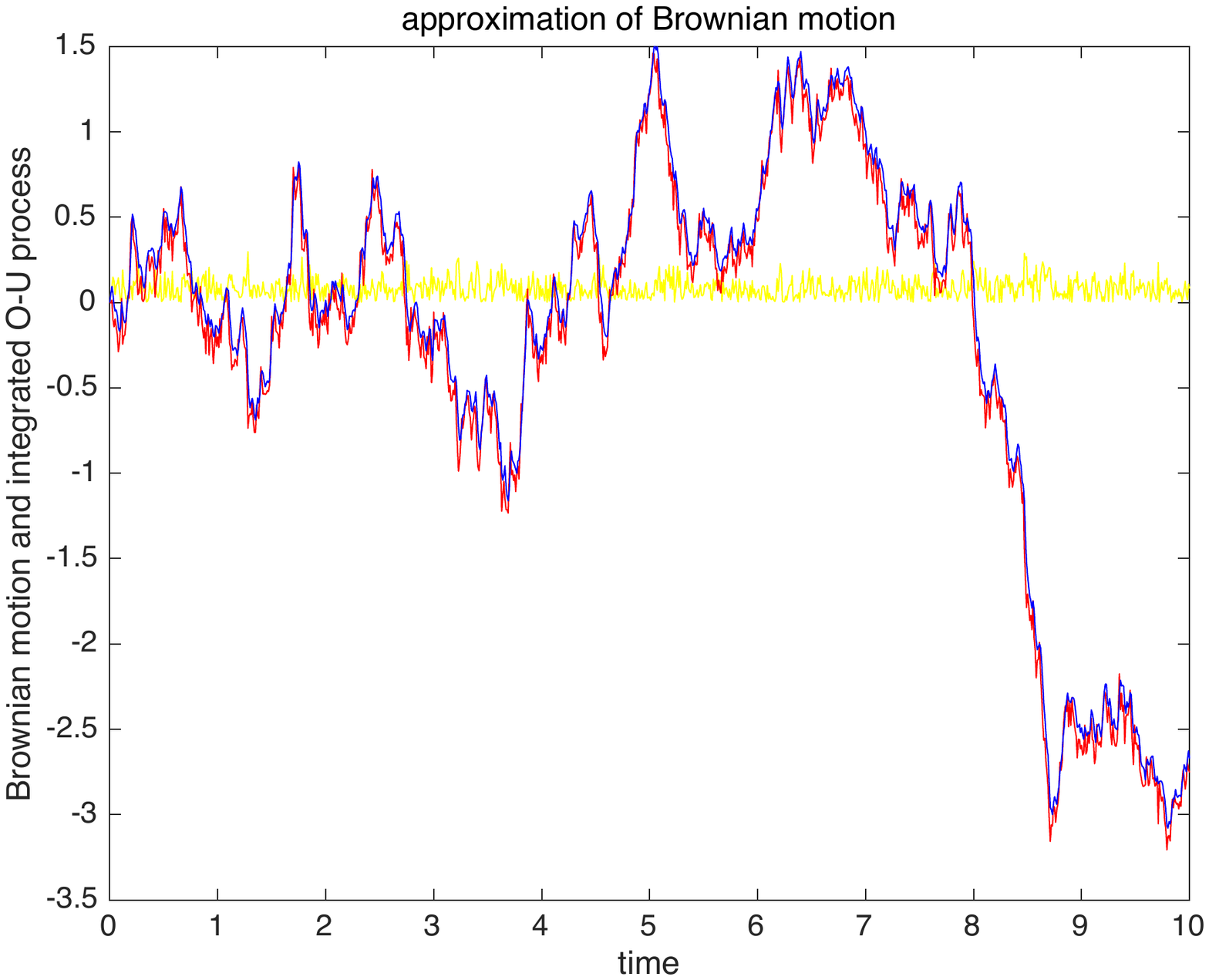}
\includegraphics[width=2in]{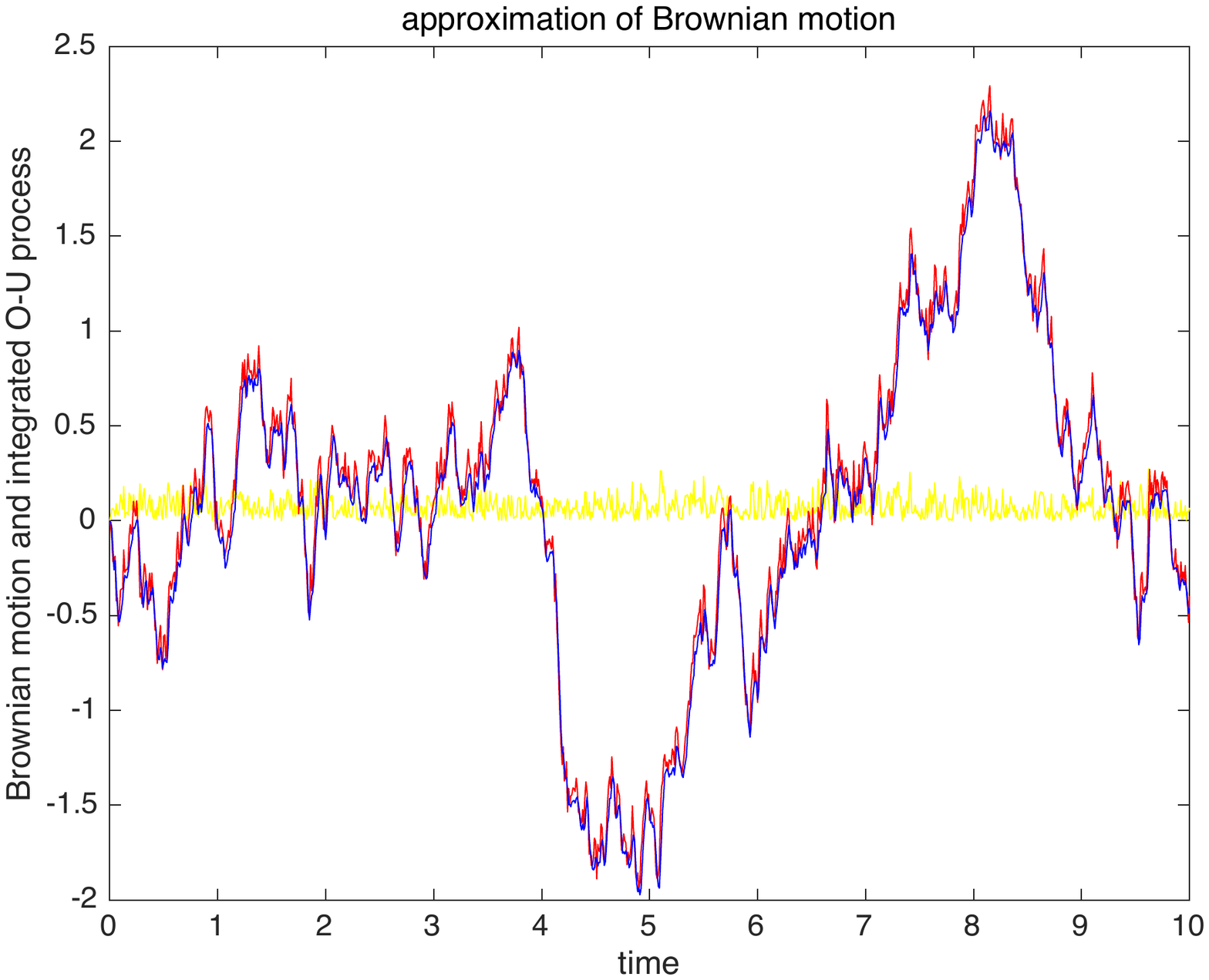}
\includegraphics[width=2in]{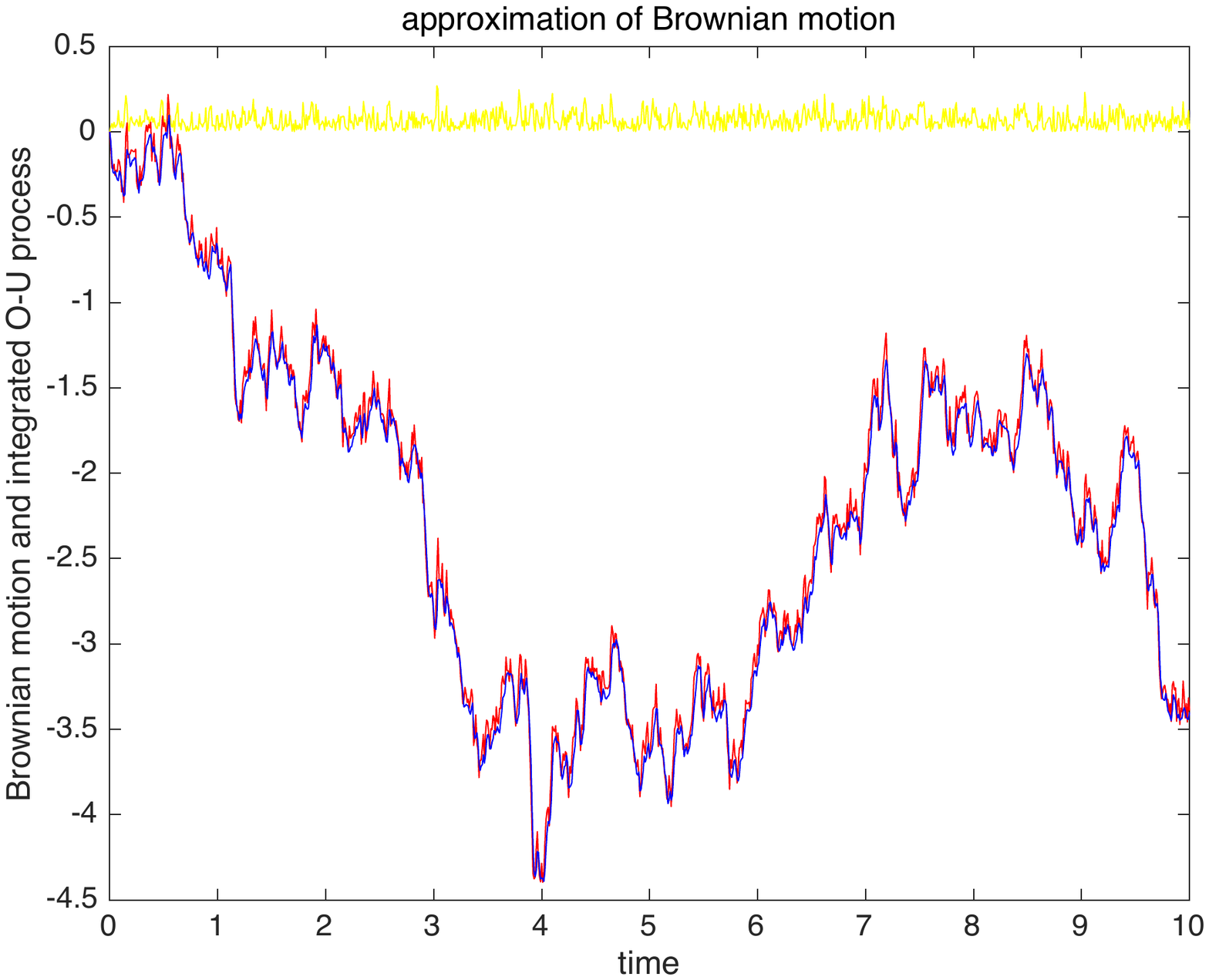}
\caption{(Color online) Three samples of Brownian motion and integrated O-U process with red and blue colour respectively. The Wong-Zakai approximation parameter $\mu=0.1$. The error is plotted with yellow line.}\label{path}
\end{figure}
\begin{figure}[H]
\centering
\includegraphics[width=2.5in]{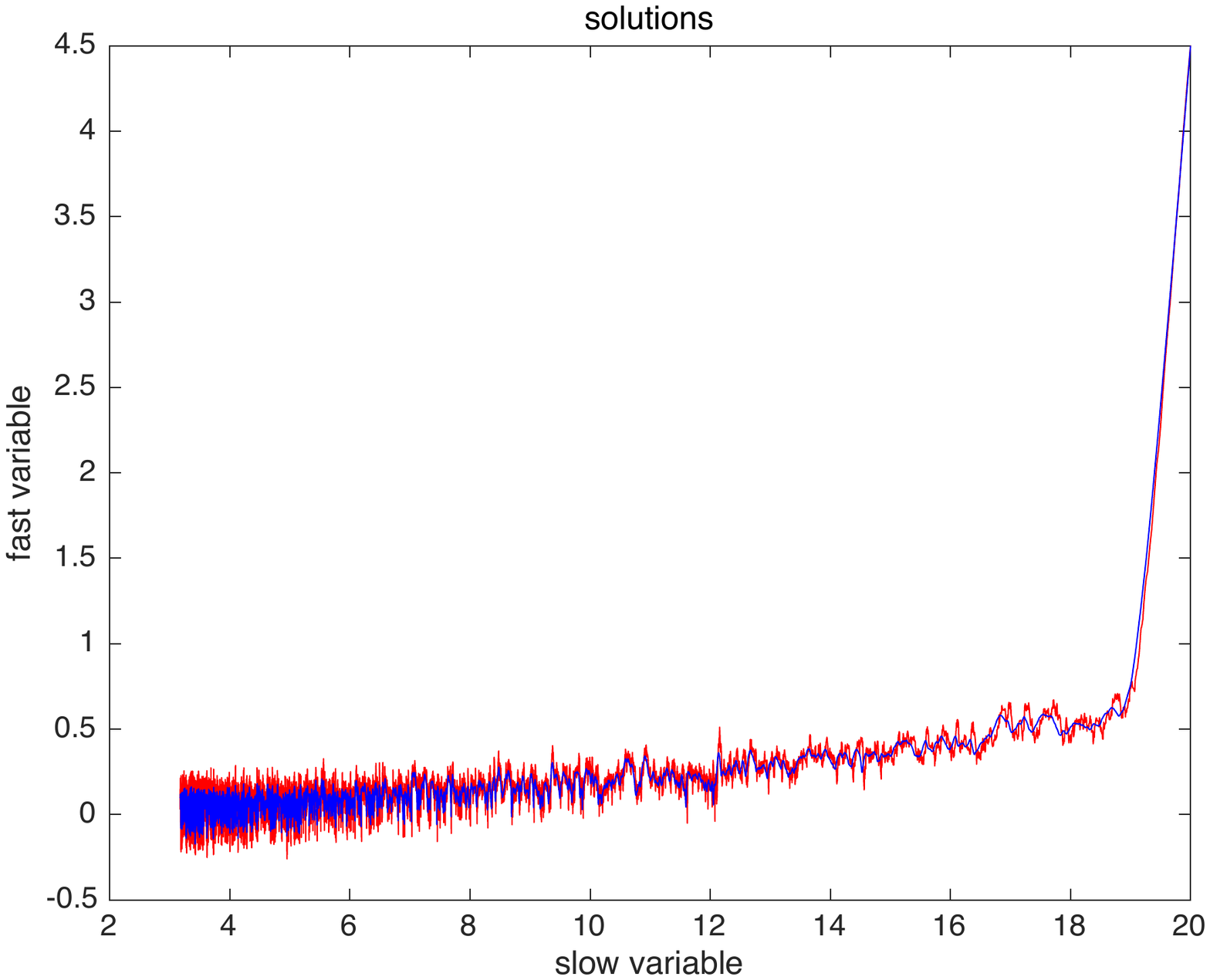}
\caption{(Color online) One Solution of SDE driven by Brownian motion in red line and the corresponding solution of the RDE by an integrated O-U process in blue line,  respectively:  the Wong-Zakai approximation parameter  $\mu=0.01$.} \label{solution}
\end{figure}

Furthermore, the solution of a SDE driven by Brownian motion can be approximated by that of a RDE driven by an integrated O-U process   \cite{P. Acquistapace, wongzakai2, wongzakai}. We shall denote the former system as the original system, and the latter as the Wong-Zakai system. It is exhibited in figure \ref{solution} by an example \ref{example} from \cite{Ren jian} used in Section \ref{sec:estimate}.

We wonder if the random slow manifold of a two-scale stochastic system driven by Brownian motion can be approximated by that of the corresponding system driven by an integrated O-U process. We shall refer to the former as the original random slow manifold and the latter as the  Wong-Zakai random slow manifold.
We consider this question in probability space $(\Omega, \mathcal{F}, \mathbb{P})$ and separable Hilbert space $(\mathbb{H},||\cdot||_{\mathbb{H}})$ with
$$\mathbb{H}=\mathbb{H}_{1}\times \mathbb{H}_{2},$$
where $(\mathbb{H}_{1},||\cdot||_{\mathbb{H}_{1}})$ and $(\mathbb{H}_{2},||\cdot||_{\mathbb{H}_{2}})$ are both separable Hilbert spaces and the norm $||\cdot||_{\mathbb{H}}=||\cdot||_{\mathbb{H}_{1}}+||\cdot||_{\mathbb{H}_{2}}$.
The original slow-fast stochastic system is
\begin{eqnarray}\label{eq:orgSDE}
\left\{\begin{array}{l}
\begin{aligned}
\dot{u}^{\varepsilon}=\frac{1}{\varepsilon}Au^{\varepsilon}+\frac{1}{\varepsilon}f(u^{\varepsilon},v^{\varepsilon})+\frac{\sigma}{\sqrt\varepsilon}\dot{B}_{t}(\omega), \qquad  \hbox{in}\,\,&\mathbb{H}_{1},\\
\dot{v}^{\varepsilon}=Bv^{\varepsilon}+g(u^{\varepsilon},v^{\varepsilon}),  ~~~~~~~~~~~~~~~~~~~~\qquad\hbox{in}\,\,&\mathbb{H}_{2}.
\end{aligned}
\end{array}
\right.
\end{eqnarray}
Here $\varepsilon$ is a small positive scale parameter. The linear operators $A, B$ and nonlinear interactions $f, g$ will be specified below. This covers slow-fast systems of SDEs  or SPDEs.

We introduce an O-U process $z^{\mu}(\theta_{t}\omega)=z^{\mu}_{t}(\omega)$ as the stationary solution of the following scalar linear SDE with a parameter $\mu$:
\begin{eqnarray}\label{O-U}
\left\{\begin{array}{l}
\dot{z}^{\mu}_{t}=-\frac{1}{\mu}z^{\mu}_{t}+\frac{1}{\mu}\dot{B}_{t}(\omega),\\
z^{\mu}_{0}=\frac{1}{\mu}\int_{-\infty}^{0}e^{\frac{s}{\mu}}\mathrm{d}B_{s}(\omega).
\end{array}
\right.
\end{eqnarray}
This O-U process is a correlated process ( `colored noise').
The integrated O-U process $\Phi^{\mu}_{t}(\omega)$ is defined as the time integration of  `colored noise', that is
\begin{eqnarray}\label{zakai}
\begin{aligned}
\Phi^{\mu}_{t}(\omega)
=&\int_{0}^{t}z^{\mu}(\theta_{s}\omega)\mathrm{d}s
=\frac{1}{\mu}\int_{0}^{t}\int_{-\infty}^{s}e^{-\frac{s-r}{\mu}}\mathrm{d}B_{r}(\omega) \mathrm{d}s.
\end{aligned}
\end{eqnarray}

The corresponding Wong-Zakai system with Brownian motion $B_t$ replaced by an integrated O-U process is
\begin{eqnarray}\label{eq:newRDE}
\left\{\begin{array}{l}
\begin{aligned}
\dot{x}^{\mu,\varepsilon}=\frac{1}{\varepsilon}Ax^{\mu,\varepsilon}+\frac{1}{\varepsilon}f(x^{\mu,\varepsilon},y^{\mu,\varepsilon})+\frac{\sigma}{\sqrt\varepsilon}\dot{\Phi}^{\mu}_{t}(\omega), \qquad~ \hbox{in}\,\, &\mathbb{H}_{1},\\
\dot{y}^{\mu,\varepsilon}=By^{\mu,{\varepsilon}}+g(x^{\mu,\varepsilon},y^{\mu,\varepsilon}), ~~~~~~~~~~~~~~~~~~~~~~\qquad\hbox{in} \,\,&\mathbb{H}_{2}.
\end{aligned}
\end{array}
\right.
\end{eqnarray}
with the same initial condition
$$w_{0}\triangleq(u_{0},v_{0})=(\eta,\xi)\triangleq\zeta,$$
$$z_{0}\triangleq(x_{0},y_{0})=(\eta,\xi)\triangleq\zeta,$$
where the linear operators $A$ and $B$ are the generators of $C_{0}-$semigroups on separable Hilbert space $\mathbb{H}_{1}$ and $\mathbb{H}_{2}$, respectively. Nonlinear functions $f$ and $g$ are continuous functions mapping from $\mathbb{H}$ to $\mathbb{H}_{1}$ and $\mathbb{H}_{2}$ respectively, with $f(0,0)=g(0,0)=0$. The scale parameter is $\varepsilon$.  The noise intensity $\sigma$ is the element of  the definition domain of $A$ satisfying $||A\sigma||\mathbb{H}_{1}<\infty$. The definition domain of $A$ is denoted by $D(A)$ and included in $\mathbb{H}_{1}$. The definition domain of $B$ is denoted by $D(B)$ and included in $\mathbb{H}_{2}$. Brownian motion $B_{t}(\omega)$ is one dimensional and considered in the canonical probability space $(C(\mathbb{R}, \mathbb{R}), \mathcal{B}(\mathbb{R},\mathbb{R}), \mathbb{P}_{B})$ (see $\cite{JinqiaoDuan}$) with compact open topology.

Both the original system and the corresponding Wong-Zakai system are slow-fast system for small $\varepsilon$.
Variables $u$ and $x$ are the fast components  in space $\mathbb{H}_{1}$, while  $v$ and $y$ are the slow components  in space $\mathbb{H}_{2}$.

\medskip
We make the following assumptions.
\begin{itemize}
\item$(A1)$ (Spectral condition)  The linear  operator A is the generator of a $C_{0}-$semigroup $e^{At}$ on $\mathbb{H}_{1}$ satisfying
$$||e^{At}x||_{\mathbb{H}_{1}}\leq e^{-\gamma_{1}t}||x||_{\mathbb{H}_{1}}, \quad t\geq0,$$
for all $x$ in $\mathbb{H}_{1}$ and some positive constant $\gamma_{1}$.\\
The linear operator $B$ is the generator of a $C_{0}-$group $e^{Bt}$ on $\mathbb{H}_{2}$ satisfying
$$||e^{Bt}y||_{\mathbb{H}_{2}}\leq e^{-\gamma_{2}t}||y||_{\mathbb{H}_{2}}, \quad t\leq0,$$
for all $y$ in $\mathbb{H}_{2}$ and some positive constant $\gamma_{2}$.

\item$(A2)$ (Lipschitz condition) The nonlinear functions $f$ and $g$ satisfy the Lipschitz condition. There exists a positive constant $K$ such that for all $(x_{i},y_{i})\in \mathbb{H}_{1}\times \mathbb{H}_{2}$
\begin{equation*}
\begin{split}
 ||f(x_{1},y_{1})-f(x_{2},y_{2})||_{\mathbb{H}_{1}}\leq K(||x_{1}-x_{2}||_{\mathbb{H}_{1}}+||y_{1}-y_{2}||_{\mathbb{H}_{2}}),\\
 ||g(x_{1},y_{1})-g(x_{2},y_{2})||_{\mathbb{H}_{2}}\leq K(||x_{1}-x_{2}||_{\mathbb{H}_{1}}+||y_{1}-y_{2}||_{\mathbb{H}_{2}}).
 \end{split}
 \end{equation*}
\item$(A3)$ (Gap condition) Assume that the Lipschitz constant $K$ of the nonlinear terms  is smaller than the decay rate $\gamma_{1}$ of $A$, that is
 $$K<\gamma_{1}.$$
 \end{itemize}
 \begin{Remark}
 The spectral condition $(A1)$ is about  the spectral set of operators $A$ and $B$,  relative to the  bounds  $\gamma_{1}$ and $\gamma_{2}$. Let $A=\Delta-\alpha I_{id}$ with domain $D(A)=H^{2}\cap H_{0}^{1}$ and some positive $\alpha$. Take $\mathbb{H}_{1}=L^{2}(D)$. Then $A$ generates a contraction semigroup $\{e^{At}:t\geq0\}$ in $\mathbb{H}_{1}$ and satisfies $(A1)$ with $\gamma_{1}=\alpha$. Let $\tilde{B}=\Delta-\beta I_{id}$ with domain $D(\tilde{B})=H^{2}\cap H_{0}^{1}$ and some positive $\beta$. Define 
\begin{displaymath}
\qquad\qquad\qquad\qquad\qquad
y=\left(\begin{array}{cc}y_{1}\\  y_{2}\end{array}\right),
B=\left(\begin{array}{cc}0 & I_{id}\\  \tilde{B} &0\end{array}\right),
\end{displaymath}
and $\mathbb{H}_{2}=H_{0}^{1}(D)\times L^{2}(D)$. Define the norm $||y||_{\mathbb{H}_{2}}=(||y_{1}||^2_{H_{0}^{1}}+||y_{2}||^2_{L^{2}})^{1/2}$. Let $D(B)=D(\tilde{B})\times H^{1}$. Then $B$ generates an unitary group in $\mathbb{H}_{2}$ satisfies conditions $(A1)$ with $\gamma_{2}=0$. More details and examples are in \cite{Chueshov, HongboFuSlowmf}.
 \end{Remark}
\begin{Remark}\label{solvesolution}
Under the preceeding  assumptions, there exists a unique solution of the original stochastic system (\ref{eq:orgSDE}). Refer to \cite{J. Duan K. Lu and B. Schmalfuss 2} and the references therein. We denote it as
$w^{\varepsilon}(t,\omega,\zeta)=(u^{\varepsilon}(t,\omega,\zeta),v^{\varepsilon}(t,\omega,\zeta))$
simply as $w^{\varepsilon}(t)=(u^{\varepsilon}(t),v^{\varepsilon}(t))$. In the mild   form, the solution solves  an integral stochastic system
\begin{equation}
\left\{\begin{array}{l}
u^{\varepsilon}(t)=e^{\frac{A}{\varepsilon}t}\eta+\frac{1}{\varepsilon}\int_{0}^{t}e^{\frac{A}{\varepsilon}(t-s)}f(u^{\varepsilon}(s),v^{\varepsilon}(s))\mathrm{d}s+\frac{1}{\sqrt{\varepsilon}}\int_{0}^{t}e^{\frac{A}{\varepsilon}(t-s)}\sigma\mathrm{d}B_{s}(\omega),\\
v^{\varepsilon}(t)=e^{Bt}\xi+\int_{0}^{t}e^{B(t-s)}g(u^{\varepsilon}(s),v^{\varepsilon}(s))\mathrm{d}s.
\end{array}
\right.
\end{equation}
The solution of Wong-Zakai system (\ref{eq:newRDE}) is denoted as $z^{\mu,\varepsilon}(t,\omega,\zeta)=(x^{\mu,\varepsilon}(t,\omega,\zeta),y^{\mu,\varepsilon}(t,\omega,\zeta))$ simply as $z^{\mu,\varepsilon}(t)=(x^{\mu,\varepsilon}(t),y^{\mu,\varepsilon}(t))$. In the mild sense, it is
\begin{equation}
\left\{\begin{array}{l}
x^{\mu,\varepsilon}(t)=e^{\frac{A}{\varepsilon}t}\eta+\frac{1}{\varepsilon}\int_{0}^{t}e^{\frac{A}{\varepsilon}(t-s)}f(x^{\mu,\varepsilon}(s),y^{\mu,\varepsilon}(s))\mathrm{d}s+\frac{1}{\sqrt{\varepsilon}}\int_{0}^{t}e^{\frac{A}{\varepsilon}(t-s)}\sigma z^{\mu}_{s}(\omega)\mathrm{d}s,\\
y^{\mu,\varepsilon}(t)=e^{Bt}\xi+\int_{0}^{t}e^{B(t-s)}g(x^{\mu,\varepsilon}(s),y^{\mu,\varepsilon}(s))\mathrm{d}s.
\end{array}
\right.
\end{equation}
\end{Remark}

\medskip
The definition of a random slow manifold is introduced in the remainder of this subsection based on references \cite{ Ludwig Arnold, JinqiaoDuan, HongboFuSlowmf, B. Schmalfuss, E.N. Lorenz}.

 Consider the canonical sample space $\Omega=C_{0}(\mathbb{R},\mathbb{R})$ and the Borel $\sigma-$algebra $\mathcal{F}=\mathcal{B}(C_{0}(\mathbb{R},\mathbb{R}))$. The sample space $\Omega$ is composed of real continuous functions which are defined on $\mathbb{R}$ and equal to zero at time $0$ . \textit{Wiener shift} $\theta_{t}$ maps the canonical sample space $C_{0}(\mathbb{R},\mathbb{R})$ into itself with $\theta_{t}\omega(s)=\omega(t+s)-\omega(t)$ for each fixed $t$ and every $s$ in $\mathbb{R}$. The distribution of $\theta_{t}\omega$ generates the Wiener measure $\mathbb{P}$, which is ergodic with respect to $\theta_{t}$.
Wiener shift $\theta_{t}$ has the following properties:
\begin{enumerate}
\item$\ \theta_{0}=id_{\Omega};$
\item$\ \theta_{t}\theta{s}=\theta_{t+s},$ for all $s,t$ in $\mathbb{R}$;
\item The mapping $\theta$ from $\mathbb{R}\times C_{0}(\mathbb{R},\mathbb{R})$ to $C_{0}(\mathbb{R},\mathbb{R})$ is measurable and $\theta_{t}\mathbb{P}=\mathbb{P}$ for all $t\in\mathbb{R}$.
\end{enumerate}
This forms a metric dynamical system $(\Omega,\mathcal{F},\mathbb{P},(\theta_{t})_{t\in\mathbb{R}})$ (refer to \cite{Ludwig Arnold}). We need Lemma $2.1$ in \cite{J. Duan K. Lu and B. Schmalfuss 2} about the properties of Brownian motion and O-U process to prove our results. Therefore we will work on the metric dynamical system used in \cite{J. Duan K. Lu and B. Schmalfuss 2} which is mentioned after Lemma $2.1$ in \cite{J. Duan K. Lu and B. Schmalfuss 2}. To make this lemma hold true, the sample space is restricted to be a subset of $C_{0}(\mathbb{R},\mathbb{R})$. This subset belongs to $\mathcal{B}(C_{0}(\mathbb{R},\mathbb{R}))$ with full measure with respect to $\mathbb{P}$ and is $(\theta_{t})_{t\in\mathbb{R}}$ invariant. The Borel $\sigma-$algebra $\mathcal{B}(C_{0}(\mathbb{R},\mathbb{R}))$ is  restricted on this subset. And the Wiener measure is restricted on the new restricted $\sigma-$algebra. We still denote this new metric dynamical system as $(\Omega,\mathcal{F},\mathbb{P},(\theta_{t})_{t\in\mathbb{R}})$.

\medskip

\begin{Definition}(Random Dynamical System)
A measurable random dynamical system on a measurable space $(\mathbb{H},||\cdot||_{\mathbb{H}})$ with Borel $\sigma$-field $\mathcal{B}(\mathbb{H})$ over the metric dynamical system
$(\Omega,\mathcal{F},\mathbb{P},(\theta_{t})_{t\in\mathbb{R}})$ is a mapping
$$\varphi : \mathbb{R}_{+}\times\Omega\times\mathbb{H}\to\mathbb{H},$$
with the following properties:
\begin{enumerate}
\item Measurability: $\varphi$ is $(\mathcal{B}(\mathbb{R}_{+})\times\mathcal{F}\times\mathcal{B}(\mathbb{H}),\,\mathcal{B}(\mathbb{H}))$-measurable.
\item Cocycle property: The mappings $\varphi(t,\omega):=\varphi(t,\omega,\cdot):\mathbb{H}\to\mathbb{H}$ form a cocycle over $\theta(\cdot)$, that is
$$\varphi(0,\omega)=id_{\mathbb{H}}\qquad \hbox{for all  }\quad\omega\in\Omega,$$
$$\varphi(t+s,\omega)=\varphi(t,\theta_{s}\omega)\circ\varphi(s,\omega)\qquad \hbox{for all }\quad s,t\in\mathbb{R}_{+},\,\,\omega\in\Omega.$$
\end{enumerate}
\end{Definition}

\begin{Definition}(Random set)\label{randomset}
A \textit{random set} is a family of nonempty closed sets  $M={M(\omega)}$ included in a metric space $(\mathbb{H},||\cdot||_{\mathbb{H}})$ satisfying the following condition: the mapping
$$\omega\to\inf_{z_{1}\in M(\omega)}||z_{1}-z_{2}||_{\mathbb{H}}$$
from $\Omega$ to $\mathbb{R}$ is a random variable for every $z_{2}\in\mathbb{H}$.
\end{Definition}
\begin{Definition}(Random slow manifold)\label{dfrsm}
A \textit{random slow manifold} of a random dynamical system $\varphi$ generated by a two-scale system is a random set $M(\omega)=\{(h(\omega,y),y):y\in \mathbb{H}_{2}\}$, satisfying the following conditions:
\begin{enumerate}
\item The random set $M(\omega)$ is invariant with respect to the random dynamical system $\varphi$, that is
$$\varphi(t,\omega,M(\omega))\subset M(\theta_{t}\omega).$$
\item The function $h(\omega,y)$ is globally Lipschitzian in $y$ for all $\omega\in\Omega$. The mapping $\omega\to h(\omega,y)$ is a random variable for any $y\in \mathbb{H}_{2}$.

\end{enumerate}
\end{Definition}
\begin{Remark}
Here the definition of the random slow manifold is based on \cite{HongboFuSlowmf, B. Schmalfuss,  E.N. Lorenz}. Inherit the opinion of Edward N. Lorenz, the random slow manifold is composed of certain initial points whose orbits evolve at a slow level rate. The difference is that the orbits through these initial points will stay in another random set with the sample of the random slow manifold replaced by its evolvement under the Wiener shift. While the deterministic slow manifold is invariant under the deterministic dynamical system, which ensures  that the whole orbit starting from one element of this set are included in the same set always. We illustrate this property in figure \ref{figinvariant}. Besides, the fast variables are the Lipschitz function of slow variables on these orbits. This guarantees that the fast variables can be controlled by slow variables. Furthermore, the graph of the random slow manifold with one fixed sample is smoother than the phase picture of solutions of stochastic system, which can be seen from figure \ref{figsm}.
\end{Remark}
Figure \ref{figinvariant} exhibits that the solution of stochastic system (\ref{exnewSDE}) starting from this curve will be on another curve $\mathcal{M}^{\varepsilon}(\theta_{t}\omega)$ at time $t$ instead of always on $\mathcal{M}^{\varepsilon}(\omega)$, which is different from the deterministic circumstance. In the deterministic case, the two curves expressed by $\mathcal{M}^{\varepsilon}(\theta_{t}\omega)$ and $\mathcal{M}^{\varepsilon}(\omega)$ are the same, because there is only one sample which implies $\theta_{t}\omega=\omega$.
One sample of the random slow manifold of system (\ref{exnewSDE}) in example \ref{example} is, a curve, the cross section taking $t=0$ of picture (a) in Figure \ref{figinvariant}. 
\begin{figure}[H]
\centering
\includegraphics[width=4in]{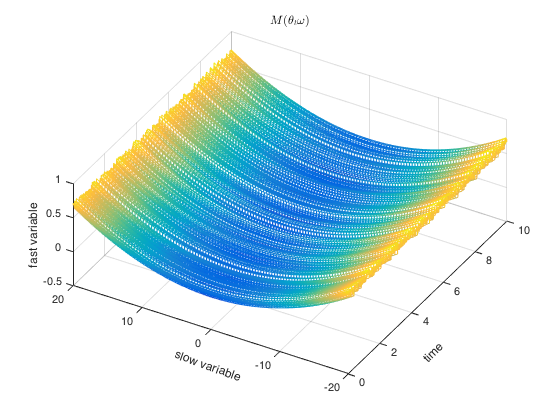}
\includegraphics[width=4in]{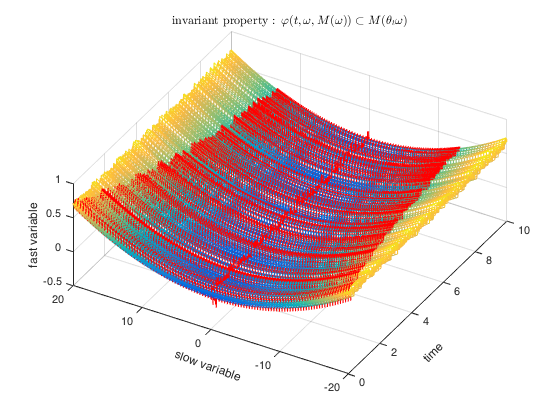}
\begin{center}(a)\hspace{8cm}(b) \end{center}
\caption{(Color online) A family of random slow manifolds $M(\theta_{t}\omega)$ with one sample $\omega$ evolving under the Wiener shift $\theta_{t}$ in picture (a). The evolution of the random slow manifold $M(\omega)$ under the random dynamical system $\varphi$ in picture (b). This shows the invariant property, $\varphi(t,\omega,M(\omega))\subset M(\theta_{t}\omega)$, in definition \ref{dfrsm}. This property is illustrated through the example \ref{example}, in which $M(\omega)=\mathcal{M}^{\varepsilon}(\omega)$ and $M(\theta_{t}\omega)=\mathcal{M}^{\varepsilon}(\theta_{t}\omega)$. The expressions of $\mathcal{M}^{\varepsilon}(\omega)$ and $\mathcal{M}^{\varepsilon}(\theta_{t}\omega)$ are computed approximately in example \ref{example}.}\label{figinvariant}
\end{figure}
\subsection{Main results}
From \cite{HongboFuSlowmf}, it is clear that the original random slow manifold of the original system (\ref{eq:orgSDE}) exists under certain conditions. We just need to find the Wong-Zakai random slow manifold of Wong-Zakai system (\ref{eq:newRDE}).
Introduce a Banach space containing continuous functions controlled by some exponential rate:
$$C^{i,-}_{\alpha}=\left\{\varphi:(-\infty,0]\mapsto H_{i}\,\hbox{ is continuous and}\,  \sup_{t\leq0}e^{-\alpha t}||\varphi(t)||_{H_{i}}<\infty \right\}$$
with the norm $||\varphi||_{C_{\alpha}^{i,-}}=\sup_{t\leq0}e^{-\alpha t}||\varphi(t)||_{H_{i}}$ for $i=1,2$. Let $C_{\alpha}^{-}$ be the product Banach space $C_{\alpha}^{-}=C_{\alpha}^{1,-}\times C_{\alpha}^{2,-}$ with the norm
$$||z||_{C_{\alpha}^{-}}=||x||_{C_{\alpha}^{1,-}}+||y||_{C_{\alpha}^{2,-}}, \quad\mbox{for } z=(x,y)\in C_{\alpha}^{-}.$$
We will show that the random orbits which are the fixed points under some operator acting on some Banach space $C_{\alpha}^{-}$ form a random set. This set is precisely the Wong-Zakai random slow manifold of Wong-Zakai system (\ref{eq:newRDE}).
\noindent
Let $\rho$ be a positive number satisfying
\begin{eqnarray}\label{cod:rho}
\gamma_{1}-\rho>K.
\end{eqnarray}
This condition is to ensure the operator we will define is contractive in the proof of the existence of the random slow manifold.
 Then the set composed of orbits in Banach space $C_{-\frac{\rho}{\varepsilon}}^{-}$ is  formulated as follows:
$$\mathcal{M}^{\mu,\varepsilon}(\omega)\triangleq \left\{\zeta\in\mathbb{H} : z^{\mu,\varepsilon}(\cdot,\omega,\zeta)\in C_{-\frac{\rho}{\varepsilon}}^{-}\right\}.$$
We will demonstrate that this family of sets is the Wong-Zakai random slow manifold in the following proposition.
Denote
$\kappa_{*}(K,\gamma_{1},\gamma_{2},\rho,\varepsilon)
=\frac{K}{\gamma_{1}-\rho}+\frac{\varepsilon K}{\rho-\varepsilon\gamma_{2}}+\frac{K^{2}}{(\gamma_{1}-\rho)(\frac{\rho}{\varepsilon}-\gamma_{2})[1-K(\frac{1}{\gamma_{1}-\rho}+\frac{\varepsilon}{\rho-\varepsilon\gamma_{2}})]}.$

\begin{Proposition}[Wong-Zakai random slow manifold]\label{OU-RSM}
If  the  assumptions $(A1)-(A3)$ hold and  the scale parameter $\varepsilon$ satisfies the condition $0<\varepsilon<\frac{\rho}{\gamma_{2}+\frac{1}{\frac{1}{K}-\frac{1}{\gamma_{1}-\rho}}}$, then Wong-Zakai system (\ref{eq:newRDE}) has a Lipschitz random slow manifold $\mathcal{M}^{\mu,\varepsilon}(\omega)$,  represented as a graph
$$\mathcal{M}^{\mu,\varepsilon}(\omega)=\left\{(h^{\mu,\varepsilon}(\omega,\xi),\xi): \xi\in \mathbb{H}_{2}\right\},\quad\mbox{for every}\quad \mu>0,$$
where $h^{\mu,\varepsilon}$ maps from $ \Omega\times \mathbb{H}_{2}$ to $\mathbb{H}_{1}$ with Lipschitz constant satisfying
$$Lip\,h^{\mu,\varepsilon,}(\omega,\cdot)\leq\frac{K}{(\gamma_{1}-\rho)[1-K(\frac{1}{\gamma_{1}-\rho}+\frac{\varepsilon}{\rho-\varepsilon\gamma_{2}})]}\,, \qquad \omega\in\Omega.$$
\end{Proposition}

\medskip

The random set $\mathcal{M}^{\mu,\varepsilon}(\omega)$ attracts all orbits exponentially. We can use this attracting property to reduce the dimension of Wong-Zakai system to that of its slow variables.

Now we can describe the main results.
As shown in \cite{HongboFuSlowmf}, when the scale parameter $\varepsilon$ is small enough, the original system (\ref{eq:orgSDE}) has an original random slow manifold
\begin{eqnarray}
\begin{aligned}
\mathcal{M}^{\varepsilon}(\omega)
&\triangleq \left\{\zeta\in\mathbb{H} : w^{\varepsilon}(\cdot,\omega,\zeta)\in C_{-\frac{\rho}{\varepsilon}}^{-}\right\}=\left\{(h^{\varepsilon}(\omega,\xi),\xi): \xi\in \mathbb{H}_{2}\right\},
\end{aligned}
\end{eqnarray}
where $h^{\varepsilon}$ maps from $ \Omega\times \mathbb{H}_{2}$ to $\mathbb{H}_{1}$, with Lipschitz constant satisfying
$$Lip\,h^{\varepsilon}(\omega,\cdot)\leq\frac{K}{(\gamma_{1}-\rho)[1-K(\frac{1}{\gamma_{1}-\rho}+\frac{\varepsilon}{\rho-\varepsilon\gamma_{2}})]}\,, \qquad \omega\in\Omega.$$
Furthermore, for every state $\zeta\in\mathbb{H}$, there exists a corresponding state $\tilde{\zeta}\in\mathcal{M}^{\varepsilon}(\omega)$, such that the original random slow manifold has the exponential tracking property
$$||w^{\varepsilon}(t,\omega,\zeta)-w^{\varepsilon}(t,\omega,\bar{\zeta})||_{\mathbb{H}}\leq C_{1}e^{-C_{2}t}||\zeta-\bar{\zeta}||_{\mathbb{H}},\quad t\geq0,\quad\omega\in\Omega,$$
with $C_{1}=\frac{1}{1-K(\frac{1}{\gamma_{1}-\rho}+\frac{\varepsilon}{\rho-\varepsilon\gamma_{2}})}$ and $C_{2}=\frac{\rho}{\varepsilon}$.
\medskip

Taking these into account, we can use Wong-Zakai random slow manifold to approximate the  original random slow manifold.

\begin{Theorem}[Wong-Zakai approximation of the original random slow manifold]\label{thm:appslmf}
If  the  assumptions $(A1)-(A3)$ hold and  the scale parameter $\varepsilon$ satisfies the condition  $0<\varepsilon<\frac{\rho}{\gamma_{2}+\frac{1}{\frac{1}{K}-\frac{1}{\gamma_{1}-\rho}}}$, then the original random slow manifold can be approximated by the Wong-Zakai random slow manifold in the following sense
$$||h^{\mu,\varepsilon}(\omega,\xi)-h^{\varepsilon}(\omega,\xi)||_{\mathbb{H}_{1}}=o(\mu^{\alpha})\quad\mbox{a.s. as }\, \mu\to 0,$$
for each $\xi\in \mathbb{H}_{2}$, $\alpha\in(0,\frac{1}{2})$.
Furthermore, this relation holds uniformly on interval $[\varepsilon_{1},\varepsilon_{2}]$ for any $\varepsilon_{1}$, $\varepsilon_{2}$ satisfying $0<\varepsilon_{1}<\varepsilon_{2}<\frac{\rho}{\gamma_{2}+\frac{1}{\frac{1}{K}-\frac{1}{\gamma_{1}-\rho}}}$.
\end{Theorem}

\begin{Theorem}[Intersystem exponential tracking property]\label{thm:intetrack}
If  the  assumptions $(A1)-(A3)$ hold and  the scale parameter $\varepsilon$ satisfies the condition $\kappa_{*}(K,\gamma_{1},\gamma_{2},\rho,\varepsilon)<1$,  then when $\mu$ tends to zero, for each initial state $\zeta\in\mathbb{H}$ of the original system (\ref{eq:orgSDE}), there is a corresponding initial state  $\bar{\zeta}$ on the Wong-Zakai random slow manifold $\mathcal{M}^{\mu,\varepsilon}(\omega)$, such that
$$||w^{\varepsilon}(t,\omega,\zeta)-z^{\mu,\varepsilon}(t,\omega,\bar{\zeta})||_{\mathbb{H}}\leq C_{1}e^{-C_{2}t}||\zeta-\bar{\zeta}||_{\mathbb{H}}+o(\mu^{\alpha}),\quad\mbox{a.s. for every}\quad t\geq0,$$
with $C_{1}=\frac{1}{1-K(\frac{1}{\gamma_{1}-\rho}+\frac{\varepsilon}{\rho-\varepsilon\gamma_{2}})}$, $C_{2}=\frac{\rho}{\varepsilon}$, and $\alpha\in(0,\frac{1}{2})$.
\end{Theorem}
With this intersystem exponential tracking property, we can reduce the original system to a lower dimensional random system pathwisely in Corollary \ref{thm:intereduce}.
\begin{Corollary}[Reduce the original system by Wong-Zakai random slow manifold]\label{thm:intereduce}
If  the  assumptions $(A1)-(A3)$ hold and  the scale parameter $\varepsilon$ satisfies the condition $\kappa_{*}(K,\gamma_{1},\gamma_{2},\rho,\varepsilon)<1$,  then when $\mu$ tends to zero,  a solution of the original system (\ref{eq:orgSDE}) can be approximated by a corresponding solution of the following system
\begin{eqnarray}\label{intesysredu}
\left\{\begin{array}{l}
\dot{\tilde{v}}^{\varepsilon}=B\tilde{v}^{\varepsilon}+g(\tilde{u}^{\varepsilon},\tilde{v}^{\varepsilon}),\\
\tilde{u}^{\varepsilon}=h^{\mu,\varepsilon}(\theta_{t}\omega,\tilde{v}^{\varepsilon}).\\
\end{array}
\right.
\end{eqnarray}
That is for each initial state $\zeta\in\mathbb{H}$ of the original system (\ref{eq:orgSDE}), there is a corresponding initial state  $\bar{\zeta}$ on the Wong-Zakai random slow manifold $\mathcal{M}^{\mu,\varepsilon}(\omega)$, such that
$$||w^{\varepsilon}(t,\omega,\zeta)-\tilde{w}^{\varepsilon}(t,\omega,\bar{\zeta})||_{\mathbb{H}}\leq C_{1}e^{-C_{2}t}||\zeta-\bar{\zeta}||_{\mathbb{H}}+o(\mu^{\alpha}),\quad\mbox{a.s. for every}\quad t\geq0,$$
with $C_{1}=\frac{1}{1-K(\frac{1}{\gamma_{1}-\rho}+\frac{\varepsilon}{\rho-\varepsilon\gamma_{2}})}$, $C_{2}=\frac{\rho}{\varepsilon}$, and $\alpha\in(0,\frac{1}{2})$.
\end{Corollary}

Based on Corollary \ref{thm:intereduce},  we will use the reduced Wong-Zakai system to estimate the unknown parameter of the original system. This offers a benefit for computational cost reduction, as we work on a lower  dimensional  system pathwisely. We will illustrate this by an example.

\section{Converting of Wong-Zakai system}\label{convert}
The random slow manifold of the original  system (\ref{eq:orgSDE}) exists by the techniques of random dynamical systems \cite{HongboFuSlowmf}. The stochastic system is transformed to a random system with equivalent dynamics. Then the investigation of the random slow manifold of the stochastic system is converted to that of the random slow manifold of the random system. In fact, Wong-Zakai system (\ref{eq:newRDE}) is already a random dynamical system. Its random slow manifold can be studied directly using the random dynamical systems theory without transformation. This is one of the benefits to do our Wong-Zakai approximation of the random slow manifold.
But in order to compare the original random slow manifold and Wong-Zakai random slow manifold more conveniently, we transform Wong-Zakai system (\ref{eq:newRDE}) to an equivalent random system with the counterpart transformation used in \cite{HongboFuSlowmf}.

Recall the transform procedure made on the original system (\ref{eq:orgSDE}) in \cite{HongboFuSlowmf} or  \cite{J. Duan K. Lu and B. Schmalfuss 2}.
Introduce a new equation
\begin{eqnarray}\label{eq:orglSDE}
\frac{\mathrm{d}\hat{u}^{\varepsilon}}{\mathrm{d}t}=\frac{1}{\varepsilon}A\hat{u}^{\varepsilon}+\frac{\sigma}{\sqrt{\varepsilon}}\dot{B}_{t}(\omega),\qquad \hat{u}^{\varepsilon}\in \mathbb{H}_{1}.
\end{eqnarray}
It has a stationary solution
$$\hat{u}^{\varepsilon}(\theta_{t}\omega)=\frac{1}{\sqrt{\varepsilon}}\int_{-\infty}^{t}e^{\frac{1}{\varepsilon}A(t-s)}\sigma\mathrm{d}B_{s}(\omega)$$
through the random variable
$$\hat{u}^{\varepsilon}(\omega)=\frac{1}{\sqrt{\varepsilon}}\int_{-\infty}^{0}e^{-\frac{1}{\varepsilon}As}\sigma\mathrm{d}B_{s}(\omega).$$
By coordinate transformation
\begin{eqnarray}\label{orgradtrs}
T^{\varepsilon}(\omega,(u^{\varepsilon},v^{\varepsilon}))=(u^{\varepsilon}-\hat{u}^{\varepsilon}(\omega),v^{\varepsilon}),\\
\end{eqnarray}
the original system $(\ref{eq:orgSDE})$ can be converted to its corresponding random  form
\begin{equation}\label{eq:orgRDS}
\left\{\begin{array}{l}
\mathrm{d}U^{\varepsilon}(t)=\frac{1}{\varepsilon}AU^{\varepsilon}(t)\mathrm{d}t+\frac{1}{\varepsilon}f(U^{\varepsilon}(t)+\hat{u}^{\varepsilon}(\theta_{t}\omega),V^{\varepsilon}(t))\mathrm{d}t,~\,\,\,U_{0}^{\varepsilon}=\eta-\hat{u}^{\varepsilon}(\omega),\\
\mathrm{d}V^{\varepsilon}(t)=BV^{\varepsilon}(t)\mathrm{d}t+g(U^{\varepsilon}(t)+\hat{u}^{\varepsilon}(\theta_{t}\omega),V^{\varepsilon}(t))\mathrm{d}t,\qquad V_{0}^{\varepsilon}=\xi.
\end{array}
\right.
\end{equation}
Then its solution in a mild sense is
 \begin{equation*}
\left\{\begin{array}{l}
U^{\varepsilon}(t)=e^{\frac{A}{\varepsilon}t}U_{0}^{\varepsilon}+\frac{1}{\varepsilon}\int_{0}^{t}e^{\frac{A}{\varepsilon}(t-s)}f(U^{\varepsilon}(s)+\hat{u}^{\varepsilon}(\theta_{s}\omega),V^{\varepsilon}(s))\mathrm{d}s,\\
V^{\varepsilon}(t)=e^{Bt}V_{0}^{\varepsilon}+\int_{0}^{t}e^{B(t-s)}g(U^{\varepsilon}(s)+\hat{u}^{\varepsilon}(\theta_{s}\omega)),V^{\varepsilon}(s)\mathrm{d}s.
\end{array}
\right.
\end{equation*}
Denote the random slow manifold of system (\ref{eq:orgRDS}) as
$M^{\varepsilon}(\omega)=\{(H^{\varepsilon}(\omega,\xi),\xi): \xi\in \mathbb{H}_{2}\},$
and through the transformation $T^{\varepsilon}$ derive the original random slow manifold of system (\ref{eq:orgSDE})
\begin{eqnarray}\label{manif:orgSDE}
\begin{aligned}
\mathcal{M}^{\varepsilon}(\omega)
=&(T^{\varepsilon})^{-1}M^{\varepsilon}(\omega)\\
=&\{(H^{\varepsilon}(\omega,\xi)+\hat{u}^{\varepsilon}(\omega),\xi):\xi\in \mathbb{H}_{2}\}\\
\triangleq&\left\{(h^{\varepsilon}(\omega,\xi),\xi):\xi\in \mathbb{H}_{2}\right\}.
\end{aligned}
\end{eqnarray}

Next, we consider the Wong-Zakai system (\ref{eq:newRDE}) with a similar procedure.
\begin{Claim}\label{Phiproperty}
Over Wiener shift $\theta_{t}$, the integrated O-U process $\Phi^{\mu}_{t}(\omega)$ satisfies the equation
\begin{eqnarray}\label{Phitheta}
\Phi^{\mu}_{s}(\theta_{t}\omega)=\Phi^{\mu}_{s+t}(\omega)-\Phi^{\mu}_{t}(\omega).
\end{eqnarray}
\end{Claim}
\medskip
\begin{proof}
Recall that
$\Phi^{\mu}_{t}(\omega)$ is defined in (\ref{zakai})
through O-U process $z^{\mu}_{t}(\omega)$  which satisfies the equation (\ref{O-U}). The solution of system (\ref{O-U}) is
$$z^{\mu}_{t}(\omega)=\frac{1}{\mu}\int_{-\infty}^{t}e^{-\frac{1}{\mu}(t-s)}\mathrm{d}B_{s}(\omega).$$
In fact, it is a stationary solution of  system (\ref{O-U}) induced by random variable
$$z_{0}^{\mu}(\omega)=\frac{1}{\mu}\int_{-\infty}^{0}e^{\frac{s}{\mu}}\mathrm{d}B_{s},$$
that is
$$z^{\mu}_{t}(\omega)=z^{\mu}_{0}(\theta_{t}\omega).$$
Using this relation   twice, we have
$$z^{\mu}_{s}(\theta_{t}\omega)=z_{0}^{\mu}(\theta_{s}(\theta_{t}\omega))=z_{0}^{\mu}(\theta_{s+t}\omega)=z^{\mu}_{s+t}(\omega).$$
The three terms in equation (\ref{Phitheta}) can be computed as
$$\Phi^{\mu}_{s}(\theta_{t}\omega)=\int_{0}^{s}z^{\mu}_{\tilde{s}}(\theta_{t}\omega)\mathrm{d}\tilde{s}=\int_{0}^{s}z^{\mu}_{\tilde{s}+t}(\omega)\mathrm{d}\tilde{s}=\int_{t}^{s+t}z^{\mu}_{\tilde{t}}(\omega)\mathrm{d}\tilde{t},$$
$$\Phi^{\mu}_{s}(\theta_{t}\omega)=\int_{t}^{s+t}z^{\mu}_{\tilde{t}}(\omega)\mathrm{d}\tilde{t}=\int_{0}^{s+t}z^{\mu}_{\tilde{t}}(\omega)\mathrm{d}\tilde{t}-\int_{0}^{t}z^{\mu}_{\tilde{t}}(\omega)\mathrm{d}\tilde{t}=\Phi^{\mu}_{s+t}(\omega)-\Phi^{\mu}_{t}(\omega).$$
Thus the claim is  proved.
\end{proof}
Similar to \cite{JinqiaoDuan}, claim \ref{Phiproperty} leads to the following lemma.
\begin{Lemma}\label{lmstasolu}
The system
\begin{eqnarray}\label{eq:newlRDE}
\frac{\mathrm{d}\hat{x}^{\mu,\varepsilon}}{\mathrm{d}t}=\frac{1}{\varepsilon}A\hat{x}^{\mu,\varepsilon}+\frac{\sigma}{\sqrt{\varepsilon}}\dot{\Phi}^{\mu}_{t}(\omega),\qquad \hat{x}^{\mu,\varepsilon}\in \mathbb{H}_{1},
\end{eqnarray}
has a stationary solution
$$\hat{x}^{\mu,\varepsilon}(\theta_{t}\omega)=\frac{1}{\sqrt{\varepsilon}}\int_{-\infty}^{t}e^{\frac{1}{\varepsilon}A(t-s)}\sigma\mathrm{d}\Phi^{\mu}_{s}(\omega),$$
through the random variable
\begin{eqnarray}\label{rvinitial}
\hat{x}^{\mu,\varepsilon}(\omega)=\frac{1}{\sqrt{\varepsilon}}\int_{-\infty}^{0}e^{-\frac{1}{\varepsilon}As}\sigma\mathrm{d}\Phi^{\mu}_{s}(\omega).
\end{eqnarray}
\end{Lemma}

\begin{proof}
The solution of system (\ref{eq:newlRDE}) is
\begin{eqnarray}\label{solutstat}
\hat{x}(t,\omega,\hat{x}_{0})=e^{\frac{1}{\varepsilon}At}\hat{x}_{0}+\frac{1}{\sqrt{\varepsilon}}\int_{0}^{t}e^{\frac{1}{\varepsilon}A(t-s)}\sigma\mathrm{d}\Phi^{\mu}_{s}(\omega).
\end{eqnarray}
On the one hand, substitute the random variable $(\ref{rvinitial})$ into the solution $(\ref{solutstat})$
\begin{equation*}
\begin{split}
\hat{x}(t,\omega,\hat{x}^{\mu,\varepsilon}(\omega))&=e^{\frac{1}{\varepsilon}At}\hat{x}^{\mu,\varepsilon}(\omega)+\frac{1}{\sqrt{\varepsilon}}\int_{0}^{t}e^{\frac{1}{\varepsilon}A(t-s)}\sigma\mathrm{d}\Phi^{\mu}_{s}(\omega)\\
&=e^{\frac{1}{\varepsilon}At}\frac{1}{\sqrt{\varepsilon}}\int_{-\infty}^{0}e^{-\frac{1}{\varepsilon}As}\sigma\mathrm{d}\Phi^{\mu}_{s}(\omega)+\frac{1}{\sqrt{\varepsilon}}\int_{0}^{t}e^{\frac{1}{\varepsilon} A(t-s)}\sigma\mathrm{d}\Phi^{\mu}_{s}(\omega)\\
&=\frac{1}{\sqrt{\varepsilon}}\int_{-\infty}^{t}e^{\frac{1}{\varepsilon}A(t-s)}\sigma\mathrm{d}\Phi^{\mu}_{s}(\omega).
 \end{split}
 \end{equation*}
 On the other hand, by differentiating expression (\ref{Phitheta}) in Claim \ref{Phiproperty} on both sides with respect to $s$, we get
$$\mathrm{d}\Phi^{\mu}_{s}(\theta_{t}\omega)=\mathrm{d}\Phi^{\mu}_{t+s}(\omega).$$
By this relation, expression $(\ref{rvinitial})$ and a variable transformation, we obtain
 \begin{eqnarray*}
 \begin{aligned}
 \hat{x}^{\mu,\varepsilon}(\theta_{t}\omega)
 &=\frac{1}{\sqrt{\varepsilon}}\int_{-\infty}^{0}e^{-\frac{1}{\varepsilon}As}\sigma\mathrm{d}\Phi^{\mu}_{s}(\theta_{t}\omega)\\
 &=\frac{1}{\sqrt{\varepsilon}}\int_{-\infty}^{0}e^{-\frac{1}{\varepsilon}As}\sigma\mathrm{d}\Phi^{\mu}_{s+t}(\omega)\\
 &=\frac{1}{\sqrt{\varepsilon}}\int_{-\infty}^{t}e^{\frac{1}{\varepsilon}A(t-s)}\sigma\mathrm{d}\Phi^{\mu}_{s}(\omega).
 \end{aligned}
 \end{eqnarray*}
We get the relation
 $$\phi(t,\omega,\hat{x}^{\mu,\varepsilon}(\omega))=\hat{x}^{\mu,\varepsilon}(\theta_{t}\omega).$$
Hence
$$\hat{x}^{\mu,\varepsilon}(\theta_{t}\omega)=\frac{1}{\sqrt{\varepsilon}}\int_{-\infty}^{t}e^{\frac{1}{\varepsilon}A(t-s)}\sigma\mathrm{d}\Phi^{\mu}_{s}(\omega)$$
 is a stationary solution for the random dynamical system (\ref{eq:newlRDE}).
\end{proof}
\begin{Remark}
 We will see the  evolution  rate for the  stationary solution $\hat{x}^{\mu,\varepsilon}(\theta_{t}\omega)$ is slower than $|t|$. One of the benefits  of this transform induced by the stationary solution $\hat{x}^{\mu,\varepsilon}(\theta_{t}\omega)$ is that it will not affect the exponential rate of the fast variables.
\end{Remark}
For the Wong-Zakai system (\ref{eq:newRDE}), the counterpart transformation is
\begin{eqnarray}\label{newradtrs}
T^{\mu,\varepsilon}(\omega,(x^{\mu,\varepsilon},y^{\mu,\varepsilon}))=(x^{\mu,\varepsilon}-\hat{x}^{\mu,\varepsilon}(\omega),y^{\mu,\varepsilon}).
\end{eqnarray}
Through transformation (\ref{newradtrs}), Wong-Zakai system (\ref{eq:newRDE}) becomes
\begin{eqnarray}\label{eq:newRDE2}
\qquad\,\,\,\left\{\begin{array}{l}
\mathrm{d}X^{\mu,\varepsilon}(t)=\frac{1}{\varepsilon}AX^{\mu,\varepsilon}(t)\mathrm{d}t+\frac{1}{\varepsilon}f(X^{\mu,\varepsilon}(t)+\hat{x}^{\mu,\varepsilon}(\theta_{t}\omega),Y^{\mu,\varepsilon}(t))\mathrm{d}t,~~X_{0}^{\mu,\varepsilon}=\eta-\hat{x}^{\mu,\varepsilon}(\omega),\\
\mathrm{d}Y^{\mu,\varepsilon}(t)=BY^{\mu,\varepsilon}(t)\mathrm{d}t+g(X^{\mu,\varepsilon}(t)+\hat{x}^{\mu,\varepsilon}(\theta_{t}\omega),Y^{\mu,\varepsilon}(t))\mathrm{d}t,~~~~~~\,Y_{0}^{\mu,\varepsilon}=\xi.
\end{array}
\right.
\end{eqnarray}
Its solution in the mild sense is
\begin{eqnarray*}
\left\{\begin{array}{l}
X^{\mu,\varepsilon}(t)=e^{\frac{A}{\varepsilon}t}X_{0}^{\mu,\varepsilon}+\frac{1}{\varepsilon}\int_{0}^{t}e^{\frac{A}{\varepsilon}(t-s)}f(X^{\mu,\varepsilon}(s)+\hat{x}^{\mu,\varepsilon}(\theta_{s}\omega),Y^{\mu,\varepsilon}(s))\mathrm{d}s,\\
Y^{\mu,\varepsilon}(t)=e^{Bt}Y_{0}^{\mu,\varepsilon}+\int_{0}^{t}e^{B(t-s)}g(X^{\mu,\varepsilon}(s)+\hat{x}^{\mu,\varepsilon}(\theta_{s}\omega)),Y^{\mu,\varepsilon}(s)\mathrm{d}s.
\end{array}
\right.
\end{eqnarray*}
Similarly, the solution mapping
\begin{eqnarray}\label{rdesolutionmapping2}
(t,\omega,Z_{0}^{\mu,\varepsilon})\to Z^{\mu,\varepsilon}(t,\omega,Z_{0}^{\mu,\varepsilon})
\end{eqnarray}
from $\mathbb{R}\times\Omega\times\mathbb{H}$ to $\mathbb{H}$
is $(\mathcal{B}(\mathbb{R})\times\mathcal{F}\times\mathcal{B}(\mathbb{H}),\,\mathcal{B}(\mathbb{H}))$-measurable, it generates a random dynamical system. Then for any $\zeta\in\mathbb{H}$ and $\omega\in\Omega$, the mapping
\begin{eqnarray}\label{solutionmapping2}
(t,\omega,\zeta)\to (T^{\mu,\varepsilon})^{-1}(\theta_{t}\omega,\cdot)\circ Z^{\mu,\varepsilon}(t,\omega,T^{\mu,\varepsilon}(\omega,\zeta))\triangleq z^{\mu,\varepsilon}(t,\omega,\zeta)
\end{eqnarray}
from $\mathbb{R}\times\Omega\times\mathbb{H}$ to $\mathbb{H}$
is a solution of Wong-Zakai system  (\ref{eq:newRDE}) by the deterministic chain rule and the definition of $\Phi^{\mu}$. Conversely,
$$(t,\omega,Z_{0}^{\mu,\varepsilon})\to T^{\mu,\varepsilon}(\theta_{t}\omega,\cdot)\circ z^{\mu,\varepsilon}(t,\omega,(T^{\mu,\varepsilon})^{-1}(\omega,Z_{0}^{\mu,\varepsilon}))$$
is a solution of system (\ref{eq:newRDE2}). Therefore the mapping (\ref{solutionmapping2}) gives all the solutions of Wong-Zakai system (\ref{eq:newRDE}).
Furthermore, the mapping (\ref{solutionmapping2}) is $(\mathcal{B}(\mathbb{R})\times\mathcal{F}\times\mathcal{B}(\mathbb{H}),\,\mathcal{B}(\mathbb{H}))$-measurable.
Consequently, Wong-Zakai system (\ref{eq:newRDE}) generates a random dynamical system defined by (\ref{solutionmapping2}). The two random dynamical systems defined by (\ref{rdesolutionmapping2}) and (\ref{solutionmapping2}) are conjugate through the invertible transform $T^{\mu,\varepsilon}$.
Ultimately, we   get the random slow manifold of system (\ref{eq:newRDE2}):
$$M^{\mu,\varepsilon}(\omega)=\{(H^{\mu,\varepsilon}(\omega,\xi),\xi): \xi\in \mathbb{H}_{2}\},$$
and through the transformation $T^{\mu,\varepsilon}$, we get  the Wong-Zakai random slow manifold for the  system (\ref{eq:newRDE}):
\begin{eqnarray}\label{manif:newRDE}
\begin{aligned}
\mathcal{M}^{\mu,\varepsilon}(\omega)
=&(T^{\mu,\varepsilon})^{-1}M^{\mu,\varepsilon}(\omega)\\
=&\{(H^{\mu,\varepsilon}(\omega,\xi)+\hat{x}^{\mu,\varepsilon}(\omega),\xi):\xi\in \mathbb{H}_{2}\}\\
\triangleq&\left\{(h^{\mu,\varepsilon}(\omega,\xi),\xi):\xi\in \mathbb{H}_{2}\right\}.
\end{aligned}
\end{eqnarray}
In fact, $\mathcal{M}^{\mu,\varepsilon}(\omega)$ is a random set  according to   Definition \ref{randomset}. Function $h^{\mu,\varepsilon}$ has the same Lipschitz condition with function $H^{\mu,\varepsilon}$.   The random set $\mathcal{M}^{\mu,\varepsilon}(\omega)$ is invariant under $z^{\mu,\varepsilon}$, as
\begin{eqnarray*}
\begin{aligned}
z^{\mu,\varepsilon}(t,\omega,\mathcal{M}^{\mu,\varepsilon}(\omega))
=&(T^{\mu,\varepsilon})^{-1}(\theta_{t}\omega,Z^{\mu,\varepsilon}(t,\omega,T^{\mu,\varepsilon}(\omega,\mathcal{M}^{\mu,\varepsilon}(\omega))))\\
=&(T^{\mu,\varepsilon})^{-1}(\theta_{t}\omega,Z^{\mu,\varepsilon}(t,\omega,M^{\mu,\varepsilon}(\omega)))\\
\subset& (T^{\mu,\varepsilon})^{-1}(\theta_{t}\omega,M^{\mu,\varepsilon}(\theta_{t}\omega)\\
=&\mathcal{M}^{\mu,\varepsilon}(\theta_{t}\omega).
\end{aligned}
\end{eqnarray*}

\section{Proof of the main results}\label{pfexwsm}
In this section, we first concentrate on the existence of the Wong-Zakai random slow manifold. 

\medskip

{\bf Proof of Proposition \ref{OU-RSM}.}
\begin{proof}
From the discussion about Wong-Zakai system $(\ref{eq:newRDE})$ in Section \ref{convert}, we only need to prove the result for converting systems.\\
Introduce operator $\mathcal{J}^{\mu,\varepsilon}$ by
\begin{eqnarray*}
\mathcal{J}^{\mu,\varepsilon}(z(\cdot)) \triangleq \begin{pmatrix}
\mathcal{J}_{1}^{\mu,\varepsilon}(z(\cdot))\\
\mathcal{J}_{2}^{\mu,\varepsilon}(z(\cdot))
\end{pmatrix},
\end{eqnarray*}
where
\begin{equation*}
\begin{split}
&\mathcal{J}_{1}^{\mu,\varepsilon}(z(\cdot))[t]=\frac{1}{\varepsilon}\int_{-\infty}^{t}e^{\frac{A(t-s)}{\varepsilon}}f(X^{\mu,\varepsilon}(s)+\hat{x}^{\mu,\varepsilon}(\theta_{s}\omega),Y^{\mu,\varepsilon}(s))\mathrm{d}s,~t\le 0,\\
&\mathcal{J}_{2}^{\mu,\varepsilon}(z(\cdot))[t]=e^{Bt}\xi+\int_{0}^{t}e^{B(t-s)}g(X^{\mu,\varepsilon}(s)+\hat{x}^{\mu,\varepsilon}(\theta_{s}\omega),Y^{\mu,\varepsilon}(s))\mathrm{d}s,~t\le 0.
\end{split}
\end{equation*}
At first, refer to \cite{J. Duan K. Lu and B. Schmalfuss, JimSDEnisis}, we prove the following integral equation exists unique solution in Banach space $C_{-\frac{\rho}{\varepsilon}}^{-}$, and it gives all the solutions of system (\ref{eq:newRDE2}) with initial value $(\eta-\hat{x}^{\mu,\varepsilon}(\omega),\xi)$.
\begin{eqnarray}\label{eq:newintsolu}
\begin{pmatrix}
X^{\mu,\varepsilon}(t)\\
Y^{\mu,\varepsilon}(t)
\end{pmatrix}
=\begin{pmatrix}
\frac{1}{\varepsilon}\int_{-\infty}^{t}e^{\frac{A(t-s)}{\varepsilon}}f(X^{\mu,\varepsilon}(s)+\hat{x}^{\mu,\varepsilon}(\theta_{s}\omega),Y^{\mu,\varepsilon}(s))\mathrm{d}s\\
e^{Bt}\xi+\int_{0}^{t}e^{B(t-s)}g(X^{\mu,\varepsilon}(s)+\hat{x}^{\mu,\varepsilon}(\theta_{s}\omega),Y^{\mu,\varepsilon}(s))\mathrm{d}s
\end{pmatrix},
\quad t\leq0.
\end{eqnarray}
Denote by $Z^{\mu,\varepsilon}(t,\omega,\xi)=(X^{\mu,\varepsilon}(t,\omega,\xi),Y^{\mu,\varepsilon}(t,\omega,\xi))$ the solution of equation \eqref{eq:newintsolu}.
For notational simplicity, we denote $X(t)=X^{\mu,\varepsilon}(t,\omega)$, $Y(t)=Y^{\mu,\varepsilon}(t,\omega)$ in the proof procedure. On the one hand, if process $(X(t),Y(t))$ is the solution of system (\ref{eq:newRDE2}) with initial value
 $(\eta-\hat{x}^{\mu,\varepsilon}(\omega),\xi)$ and is in Banach space $C_{-\frac{\rho}{\varepsilon}}^{-}$. For $\tau\leq t\leq0$, applying the integrating factor method to system (\ref{eq:newRDE2}),
\begin{eqnarray}\label{solu:pi2}
X(t)=e^{\frac{1}{\varepsilon}At}\Big(e^{-\frac{1}{\varepsilon}A\tau}X(\tau)+\frac{1}{\varepsilon}\int_{\tau}^{t}e^{-\frac{1}{\varepsilon}As}f(X(s)+\hat{x}^{\mu,\varepsilon}(\theta_{s}\omega),Y(s))\mathrm{d}s\Big).
\end{eqnarray}
\begin{eqnarray}
Y(t)=e^{B(t-\tau)}Y(\tau)+e^{Bt}\int_{\tau}^{t}e^{-Bs}g(X(s)+\hat{x}^{\mu,\varepsilon}(\theta_{s}\omega),Y(s))\mathrm{d}s.
\end{eqnarray}
With the fact that $\hat{x}^{\mu,\varepsilon}(\theta_{s}\omega)\in C_{-\frac{\rho}{\varepsilon}}^{1,-}$, we can check the form of $Y(t)$ is bounded under $||\cdot||_{C_{-\frac{\rho}{\varepsilon}}^{2,-}}$ by setting $\tau=0$. In fact
\begin{eqnarray*}
\begin{aligned}
||Y(t)||_{C_{-\frac{\rho}{\varepsilon}}^{2,-}}
\leq&\sup_{t\leq0}e^{(\frac{\rho}{\varepsilon}-\gamma_{2})t}\Big(||\xi||_{\mathbb{H}_{2}}+\int_{t}^{0}e^{\gamma_{2}s}K(||X(s)+\hat{x}^{\mu,\varepsilon}(\theta_{s}\omega)||_{\mathbb{H}_{1}}+||Y(s)||_{\mathbb{H}_{2}})\mathrm{d}s\Big)\\
=&||\xi||_{\mathbb{H}_{2}}+\frac{K}{\frac{\rho}{\varepsilon}-\gamma_{2}}(||X(s)||_{C_{-\frac{\rho}{\varepsilon}}^{1,-}}+||Y(s)||_{C_{-\frac{\rho}{\varepsilon}}^{2,-}}+||\hat{x}^{\mu,\varepsilon}(\theta_{s}\omega)||_{C_{-\frac{\rho}{\varepsilon}}^{1,-}})
<\infty.
\end{aligned}
\end{eqnarray*}
To find the special form for $X(t)$ such that $(X(t),Y(t))$ is in the space $C_{-\frac{\rho}{\varepsilon}}^{-}$, and notice that
$$||X(\cdot)||_{C_{-\frac{\rho}{\varepsilon}}^{1,-}}=\sup_{t\leq0}e^{\frac{\rho}{\varepsilon}t}||e^{\frac{1}{\varepsilon}At}\big(e^{-\frac{1}{\varepsilon}A\tau}X(\tau)+\frac{1}{\varepsilon}\int_{\tau}^{t}e^{-\frac{1}{\varepsilon}As}f(X(s)+\hat{x}^{\mu,\varepsilon}(\theta_{s}\omega),Y(s))\mathrm{d}s\big)||_{\mathbb{H}_{1}}.$$
Then the following inequality holds for $t\leq0$. Let $t\to-\infty$, we get
\begin{equation*}
\begin{aligned}
&||\big(e^{-\frac{1}{\varepsilon}A\tau}X(\tau)+\frac{1}{\varepsilon}\int_{\tau}^{t}e^{-\frac{1}{\varepsilon}As}f(X(s)+\hat{x}^{\mu,\varepsilon}(\theta_{s}\omega),Y(s))\mathrm{d}s\big)||_{\mathbb{H}_{1}}
\leq e^{\frac{\gamma_{1}}{\varepsilon}t}e^{-\frac{\rho}{\varepsilon}t}||X(\cdot)||_{C_{-\frac{\rho}{\varepsilon}}^{1,-}}\to0.
\end{aligned}
\end{equation*}
Hence for $t\leq0$,
$$X(\tau)=-\frac{1}{\varepsilon}e^{\frac{1}{\varepsilon}A\tau}\int_{\tau}^{t}e^{-\frac{1}{\varepsilon}As}f(X(s)+\hat{x}^{\mu,\varepsilon}(\theta_{s}\omega),Y(s))\mathrm{d}s.$$
Let $t\to-\infty$ on both sides of the above expression, and replacing time variable $\tau$ by $t$, we get
$$X(t)=\frac{1}{\varepsilon}\int_{-\infty}^{t}e^{\frac{A}{\varepsilon}(t-s)}f(X(s)+\hat{x}^{\mu,\varepsilon}(\theta_{s}\omega),Y(s))\mathrm{d}s.$$
Hence if $(X(t),Y(t))$ is a solution  of system (\ref{eq:newRDE2}) in the space $C_{-\frac{\rho}{\varepsilon}}^{-}$ with initial value $(\eta-\hat{x}^{\mu,\varepsilon}(\omega),\xi)$, then it can be written as in \eqref{eq:newintsolu}.\\
On the other hand, when process $Z(t)=(X(t),Y(t))$ is written in form \eqref{eq:newintsolu} and belong to Banach space $C_{-\frac{\rho}{\varepsilon}}^{-}$, it is easy to deduce that $Z(t)$ solves the system (\ref{eq:newRDE2}) by a direct computation.\\
Next, we can prove the operator $\mathcal{J}^{\mu,\varepsilon}$ maps $C^{-}_{-\frac{\rho}{\varepsilon}}$ into itself. Note that for each fixed $\mu>0$,
$$\hat{x}^{\mu,\varepsilon}(\theta_{t}\omega)\in C_{-\frac{\rho}{\varepsilon}}^{1,-}.$$
In fact, refer to Lemma $2.1$ in \cite{J. Duan K. Lu and B. Schmalfuss 2} about the properties of Brownian motion and O-U process
$$\frac{z^{\mu}(\theta_{t}\omega)}{t}\to0,\quad\mbox{as}\quad t\to\pm\infty,\quad\mbox{for each fixed}\quad \mu>0.$$

For a positive number $M$, we can find a positive number $T$ such that
$$\frac{|z^{\mu}(\theta_{t}\omega)|}{|t|}<M\quad\mbox{for}\quad t\in(-\infty,-T).$$
We can deduce
\begin{eqnarray*}
\begin{aligned}
\sup_{t\leq0}e^{\frac{\rho}{\varepsilon}t}||\hat{x}^{\mu}(\theta_{t}\omega)||_{\mathbb{H}_{1}}
=&\sup_{t\leq0}e^{\frac{\rho}{\varepsilon}t}||\int_{-\infty}^{t}\frac{1}{\sqrt{\varepsilon}}e^{\frac{1}{\varepsilon}A(t-s)}\sigma\mathrm{d}\Phi^{\mu}_{s}(\omega)||_{\mathbb{H}_{1}}\\
\leq&\sup_{t\leq0}e^{\frac{\rho}{\varepsilon}t}\int_{-\infty}^{t}\frac{||\sigma||_{\mathbb{H}_{1}}}{\sqrt{\varepsilon}}|s| e^{-\frac{\gamma_{1}}{\varepsilon}(t-s)}\frac{|z^{\mu}(\theta_{s}\omega)|}{|s|}\mathrm{d}s\\
\leq&\sup_{t\leq0}(M\frac{||\sigma||_{\mathbb{H}_{1}}}{\sqrt{\varepsilon}}\frac{\varepsilon}{\gamma_{1}}(\frac{\varepsilon}{\gamma_{1}}+T)e^{\frac{\rho}{\varepsilon}t}+M\frac{||\sigma||_{\mathbb{H}_{1}}}{\sqrt{\varepsilon}}\frac{\varepsilon}{\gamma_{1}})\\
<&\infty.
\end{aligned}
\end{eqnarray*}
Taking $z=(x,y)\in C^{-}_{-\frac{\rho}{\varepsilon}}$, we have
$$||\mathcal{J}^{\mu,\varepsilon}||_{C_{-\frac{\rho}{\varepsilon}}}\leq\kappa(K,\gamma_{1},\gamma_{2},\rho,\varepsilon)(||z||_{C_{-\frac{\rho}{\varepsilon}}^{-}}+||\hat{x}^{\mu,\varepsilon}(\theta_{s}\omega)||_{C_{-\frac{\rho}{\varepsilon}}^{1,-}})+||\xi||_{\mathbb{H}_{2}}$$
with
\begin{eqnarray}\label{param:kappa}
\kappa(K,\gamma_{1},\gamma_{2},\rho,\varepsilon)=\frac{K}{\gamma_{1}-\rho}+\frac{\varepsilon K}{\rho-\varepsilon\gamma_{2}}.
\end{eqnarray}
The conclusion comes from
similar arguments leading to $\mathcal{J}^{\mu,\varepsilon}$ maps $C_{-\frac{\rho}{\varepsilon}}^{-}$ into itself.

Furthermore 
$$||\mathcal{J}^{\mu,\varepsilon}(z)-\mathcal{J}^{\mu,\varepsilon}(\bar{z})||_{C_{-\frac{\rho}{\varepsilon}}^{-}}\leq\kappa_{1}(K,\gamma_{1},\gamma_{2},\rho,\varepsilon)||z-\bar{z}||_{C_{-\frac{\rho}{\varepsilon}}^{-}}$$
with
$$\kappa_{1}(K,\gamma_{1},\gamma_{2},\rho,\varepsilon)=\frac{K}{\gamma_{1}-\rho}+\frac{\varepsilon K}{\rho-\varepsilon\gamma_{2}}\,\to\frac{K}{\gamma_{1}-\rho},\qquad \varepsilon\to0.$$
According to inequality (\ref{cod:rho}), there exists a sufficiently small positive constant $\varepsilon_{0}$ satisfying $\varepsilon_{0}<\frac{\rho}{\gamma_{2}+\frac{1}{\frac{1}{K}-\frac{1}{\gamma_{1}-\rho}}}$, such that
$$\kappa_{1}(K,\gamma_{1},\gamma_{2},\rho,\varepsilon)<1,\quad\hbox{for every}\,\,\varepsilon\hbox{ in }(0,\varepsilon_{0}].$$
That means mapping $\mathcal{J}^{\mu,\varepsilon}$ is strictly contractive in $C_{-\frac{\rho}{\varepsilon}}^{-}$. Hence the integral equation has a unique solution $Z^{\mu,\varepsilon}(t,\omega,\xi)=(X^{\mu,\varepsilon}(t,\omega,\xi),Y^{\mu,\varepsilon}(t,\omega,\xi))\quad\hbox{in}\,\,C_{-\frac{\rho}{\varepsilon}}^{-}.$\\
Define
$$H^{\mu,\varepsilon}(\omega,\xi)=\frac{1}{\varepsilon}\int_{-\infty}^{0}e^{-\frac{As}{\varepsilon}}f(X^{\mu,\varepsilon}(s)+\hat{x}^{\mu,\varepsilon}(\theta_{s}\omega),Y^{\mu,\varepsilon}(s))\mathrm{d}s,$$
then for all $\omega$ in $\Omega$ and $\xi_{1},\xi_{2}$ in $\mathbb{H}_{2}$,
\begin{eqnarray}\label{res:thsmlip}
||H^{\mu,\varepsilon}(\omega,\xi_{1})-H^{\mu,\varepsilon}(\omega,\xi_{2})||_{\mathbb{H}_{1}}\leq\frac{K}{\gamma_{1}-\rho}\frac{1}{1-\kappa(K,\gamma_{1},\gamma_{2},\rho,\varepsilon)}||\xi_{1}-\xi_{2}||_{\mathbb{H}_{2}}.
\end{eqnarray}
It follows that
\begin{eqnarray}
M^{\mu,\varepsilon}(\omega)
\triangleq\{\xi:Z^{\mu,\varepsilon}(t,\omega,\xi)\in C_{-\frac{\rho}{\varepsilon}}^{-}\}
=\{(H^{\mu,\varepsilon}(\omega,\xi),\xi): \xi\in \mathbb{H}_{2}\}.
\end{eqnarray}
In fact, it is a random set. For this, we need to show that for every  $z=(x,y)$ in $\mathbb{H}$,  the mapping
$$\omega\mapsto\inf_{z'\in H}||(x,y)-(H^{\mu,\varepsilon}(\omega,\mathbb{P}z'),\mathbb{P}z')||$$
is measurable. To this end, we need Theorem
\uppercase\expandafter{\romannumeral3}.9 in \cite{CCastaing} about the properties of mutifunctions.
Right here $(T,\mathcal{T})=(\Omega,\mathcal{F})$,  $t=\omega$, $X=\mathbb{H}$, and
$$\Gamma(\omega)=\{(H^{\mu,\varepsilon}(\omega,\mathcal{P}z'),\mathcal{P}z'); z'\in\mathbb{H}\}.$$
There exists a countable dense subset $\mathbb{H}_{c}$ due to the separability of Hilbert space $\mathbb{H}$.  Without loss of generality, we assume $\mathbb{H}_{c}=\{z'_{1},\cdots,z'_{n}\cdots\}$.\\
Let
$$\sigma_{n}(\omega)=\{(H^{\mu,\varepsilon}(\omega,\mathcal{P}z'_{n}),\mathcal{P}z'_{n})\},\quad z'_{n}\in\mathbb{H}_{c},$$
then $\Gamma(\omega)=\overline{\{\sigma_{n}(\omega)\}}$
and $(\sigma_{n}(\omega))$ is a sequence of measurable selections of $\Gamma(\omega)$. The third property of \uppercase\expandafter{\romannumeral3}.9 in \cite{CCastaing} holds true, so the second property holds true too. That is, for   $z=(x,y)\in\mathbb{H}=\mathbb{H}_{1}\times \mathbb{H}_{2}$,
$$d(z,\Gamma(\cdot))=\inf_{z'\in\mathbb{H}}||(x,y)-(H^{\mu,\varepsilon}(\omega,\mathcal{P}z'),\mathcal{P}z')||$$
as a function of $\omega$ is measurable, which implys that  $M^{\mu,\varepsilon}(\omega)$ is a random set.
If we can show that $M^{\mu,\varepsilon}(\omega)$ is invariant, then $M^{\mu,\varepsilon}(\omega)$ is the Lipschitz random slow manifold of the random dynamical system of system (\ref{eq:newRDE2}), with Lipschitz constant given by inequality (\ref{res:thsmlip}).

It is similar with \cite{HongboFuSlowmf} to show that $M^{\mu,\varepsilon}(\omega)$ is invariant. That is, for every $Z_{0}^{\mu,\varepsilon}=(\eta+\hat{x}^{\mu,\varepsilon}(\omega),\xi)$ in $M^{\mu,\varepsilon}(\omega)$, $Z^{\mu,\varepsilon}(s,\omega,Z_{0}^{\mu,\varepsilon})$ belongs to $M^{\mu,\varepsilon}(\theta_{s}\omega)$ for all positive time $s$.
\end{proof}

Now we use Wong-Zakai random slow manifold obtained in Proposition \ref{OU-RSM} to approximate the original random slow manifold of the original system (\ref{eq:orgSDE}). Ultimately, we use Wong-Zakai random slow manifold to reduce the original system (\ref{eq:orgSDE}).

\medskip

\noindent{\bf Proof of Theorem \ref{thm:appslmf}}
\begin{proof}
Under conditions of Theorem \ref{thm:appslmf}, the original random slow manifold and Wong-Zakai random slow manifold exist and have exact expression. 
$$h^{\mu,\varepsilon}(\omega,\xi)=\frac{1}{\varepsilon}\int_{-\infty}^{0}e^{-\frac{A}{\varepsilon}s}f(x^{\mu,\varepsilon},y^{\mu,\varepsilon})\mathrm{d}s+\hat{x}^{\mu,\varepsilon}(\omega),$$
$$h^{\varepsilon}(\omega,\xi)=\frac{1}{\varepsilon}\int_{-\infty}^{0}e^{-\frac{A}{\varepsilon}s}f(u^{\varepsilon},v^{\varepsilon}))\mathrm{d}s+\hat{u}^{\varepsilon}(\omega).$$
We estimate the difference between them,
\begin{eqnarray}\label{ine2}
\begin{aligned}
&||h^{\mu,\varepsilon}(\omega,\xi)-h^{\varepsilon}(\omega,\xi)||_{\mathbb{H}_{1}}\\
\leq&\frac{1}{\varepsilon}||\int_{-\infty}^{0}e^{-\frac{A}{\varepsilon}s}(f(x^{\mu,\varepsilon},y^{\mu,\varepsilon})-f(u^{\varepsilon},v^{\varepsilon}))\mathrm{d}s||_{\mathbb{H}_{1}}+||\hat{x}^{\mu,\varepsilon}(\omega)-\hat{u}^{\varepsilon}(\omega)||_{\mathbb{H}_{1}}\\
=&\frac{K}{\gamma_{1}-\rho}||z^{\mu,\varepsilon}-w^{\varepsilon}||_{C_{-\frac{\rho}{\varepsilon}}^{-}}+||\hat{x}^{\mu,\varepsilon}(\theta_{s}\omega)-\hat{u}^{\varepsilon}(\theta_{s}\omega)||_{C_{-\frac{\rho}{\varepsilon}}^{1,-}}.
\end{aligned}
\end{eqnarray}
For the first term
\begin{eqnarray}\label{ine}
\begin{aligned}
||z^{\mu,\varepsilon}(t)-w^{\varepsilon}(t)||_{C_{-\frac{\rho}{\varepsilon}}^{-}}
\leq&||Z^{\mu,\varepsilon}(t)-W^{\varepsilon}(t)||_{C_{-\frac{\rho}{\varepsilon}}^{-}}+||\hat{x}^{\mu,\varepsilon}(\theta_{t}\omega)-\hat{u}^{\varepsilon}(\theta_{t}\omega)||_{C_{-\frac{\rho}{\varepsilon}}^{1,-}}\\
\leq&\frac{1}{1-\kappa}||\hat{x}^{\mu,\varepsilon}(\theta_{s}\omega)-\hat{u}^{\varepsilon}(\theta_{s}\omega)||_{C_{-\frac{\rho}{\varepsilon}}^{1,-}},
\end{aligned}
\end{eqnarray}
where $Z^{\mu,\varepsilon}(t)=z^{\mu,\varepsilon}(t)-(\hat{x}^{\mu,\varepsilon}(\theta_{t}\omega),0)$ and $W^{\varepsilon}(t)=w^{\varepsilon}(t)-(\hat{u}^{\varepsilon}(\theta_{t}\omega),0)$.\\
In fact,
\begin{eqnarray*}
\begin{aligned}
||X^{\mu,\varepsilon}-U^{\varepsilon}||_{C_{-\frac{\rho}{\varepsilon}}^{1,-}}
\leq&\sup_{t\leq0}\frac{1}{\varepsilon}\int_{-\infty}^{t}e^{-\frac{\gamma_{1}}{\varepsilon}(t-s)}e^{\frac{\rho}{\varepsilon}s}(||Z^{\mu,\varepsilon}-W^{\varepsilon}||_{C_{-\frac{\rho}{\varepsilon}}^{-}}+||\hat{x}^{\mu,\varepsilon}(\theta_{s}\omega)-\hat{u}^{\varepsilon}(\theta_{s}\omega)||_{C_{-\frac{\rho}{\varepsilon}}^{1,-}})\mathrm{d}s\\
=&\frac{K}{\gamma_{1}-\rho}(||Z^{\mu,\varepsilon}-W^{\varepsilon}||_{C_{-\frac{\rho}{\varepsilon}}^{-}}+||\hat{x}^{\mu,\varepsilon}(\theta_{s}\omega)-\hat{u}^{\varepsilon}(\theta_{s}\omega)||_{C_{-\frac{\rho}{\varepsilon}}^{1,-}}),
\end{aligned}
\end{eqnarray*}
and
\begin{eqnarray*}
\begin{aligned}
||Y^{\mu,\varepsilon}-V^{\varepsilon}||_{C_{-\frac{\rho}{\varepsilon}}^{2,-}}
\leq&\sup_{t\leq0}\int_{t}^{0}e^{\frac{\rho}{\varepsilon}t}e^{-\gamma_{2}(t-s)}Ke^{-\frac{\rho}{\varepsilon}s}(||Z^{\mu,\varepsilon}-W^{\varepsilon}||_{C_{-\frac{\rho}{\varepsilon}}^{-}}+||\hat{x}^{\mu,\varepsilon}(\theta_{s}\omega)-\hat{u}^{\varepsilon}(\theta_{s}\omega)||_{C_{-\frac{\rho}{\varepsilon}}^{1,-}})\mathrm{d}s\\
=&K\frac{\varepsilon}{\rho-\varepsilon\gamma_{2}}(||Z^{\mu,\varepsilon}-W^{\varepsilon}||_{C_{-\frac{\rho}{\varepsilon}}^{-}}+||\hat{x}^{\mu,\varepsilon}(\theta_{s}\omega)-\hat{u}^{\varepsilon}(\theta_{s}\omega)||_{C_{-\frac{\rho}{\varepsilon}}^{1,-}}).
\end{aligned}
\end{eqnarray*}
We get
 $$||Z^{\mu,\varepsilon}-W^{\varepsilon}||_{C_{-\frac{\rho}{\varepsilon}}^{-}}\leq\frac{\kappa}{1-\kappa}||\hat{x}^{\mu,\varepsilon}(\theta_{s}\omega)-\hat{u}^{\varepsilon}(\theta_{s}\omega)||_{C_{-\frac{\rho}{\varepsilon}}^{1,-}}.$$
It is enough to consider the difference $||\hat{x}^{\mu,\varepsilon}(\theta_{s}\omega)-\hat{u}^{\varepsilon}(\theta_{s}\omega)||_{C_{-\frac{\rho}{\varepsilon}}^{1,-}}$.\\
Similarly to $\hat{x}^{\mu,\varepsilon}(\theta_{t}\omega)\in C_{-\frac{\rho}{\varepsilon}}^{1,-}$ which has been
proven in the proof procedure of Proposition \ref{OU-RSM}, we can get $\hat{u}^{\varepsilon}(\theta_{t}\omega)\in C_{-\frac{\rho}{\varepsilon}}^{1,-}$ and so is their difference.
Based on Lemma $2.1$ in \cite{J. Duan K. Lu and B. Schmalfuss 2} about the properties of Brownian motion and O-U process, we have
$$\frac{|B_{t}(\omega)|}{|t|}\to0,\quad\mbox{as}\quad t\to\pm\infty,$$
$$\frac{|\Phi^{\mu}_{t}(\omega)|}{|t|}=\frac{|\int_{0}^{t}z^{\mu}(\theta_{s}\omega)\mathrm{d}s|}{|t|}\to0,\quad\mbox{as}\quad t\to\pm\infty,\quad\mbox{uniformly for}\quad 0<\mu\ll1.$$
These properties can infer
$$\frac{|\Phi^{\mu}_{t}(\omega)-B_{t}(\omega)|}{|t|}\to0,\quad\mbox{as}\quad t\to\pm\infty,\quad\mbox{uniformly for}\quad 0<\mu\ll1.$$
The difference
\begin{eqnarray*}
\begin{aligned}
&||\hat{x}^{\mu,\varepsilon}(\theta_{t}\omega)-\hat{u}^{\varepsilon}(\theta_{t}\omega)||_{C_{-\frac{\rho}{\varepsilon}}}^{1,-}\\
=&\sup_{t\leq0}||e^{\frac{\rho}{\varepsilon}t}\frac{1}{\sqrt{\varepsilon}}\int_{-\infty}^{t}e^{\frac{1}{\varepsilon}A(t-s)}\sigma\mathrm{d}(\Phi_{s}^{\mu}(\omega)-B_{s}(\omega))||_{\mathbb{H}_{1}}\\
\leq&\sup_{t<0}|t|e^{\frac{\rho}{\varepsilon}t}\frac{||\sigma||_{\mathbb{H}_{1}}}{\sqrt{\varepsilon}}\frac{|\Phi_{t}^{\mu}-B_{t}|}{|t|}+\sup_{t<0}e^{\frac{\rho}{\varepsilon}t}\frac{||A\sigma||_{\mathbb{H}_{1}}}{\varepsilon\sqrt{\varepsilon}}\int_{-\infty}^{t}|s|e^{-\frac{\gamma_{1}}{\varepsilon}(t-s)}\frac{|\Phi_{s}^{\mu}-B_{s}|}{|s|}\mathrm{d}s\\
\triangleq&I_{1}+I_{2}.
\end{aligned}
\end{eqnarray*}
For every $\delta>0(\ll1)$, there exist a number $T$ such that
$$\frac{|\Phi_{t}^{\mu}-B_{t}|}{|t|}<\delta,\quad\mbox{and}\quad|t|e^{\frac{\rho}{\varepsilon}t}<\delta\quad\mbox{for}\quad -\infty\leq t\leq-T.$$
Referring to \cite{S. Al-azzawi}, for $t\in[-T,0]$ we have
$$\Phi^{\mu}_{t}-B_{t}=\mu z^{\mu}(\omega)-\mu z^{\mu}(\theta_{t}\omega).$$
Based on Kolmogorov's continuity criterion,
\begin{eqnarray}\label{Phiapproximation}
|\Phi_{t}^{\mu}-B_{t}|=o(\mu^{\alpha})\quad \mbox{a.s. for } \alpha\in(0,\frac{1}{2}).
\end{eqnarray}
Due to the arbitrary of $\delta$, for every $\varepsilon>0$ and $\mu$ sufficiently small
$$I_{1}=o(\mu^{\alpha})\quad \mbox{a.s.}.$$
Take number $\beta\in(\gamma_{1}-\rho,\gamma_{1})$ fixed. Then for each $\mu$ fixed, there exist a sufficiently large positive number $T_{2}$ such that $e^{\frac{\gamma_{1}-\beta}{\varepsilon}s}|B_{s}-\Phi_{s}^{\mu}|<C(\omega)\mu$, for $s\in(-\infty,-T_{2})$ and all $\omega\in\Omega$.
For this $T_{2}>0$, there exist a sufficiently small number $\mu_{0}>0$ such that $|B_{s}-\Phi_{s}^{\mu}|\leq\mu^{\alpha}$ a.s. for $s\in[-T_{2},0)$, $\mu\in(0,\mu_{0})$. 
$$\frac{||A\sigma||_{\mathbb{H}_{1}}}{\varepsilon\sqrt{\varepsilon}}e^{\frac{\rho}{\varepsilon}t}\int_{-T_{2}}^{t}e^{-\frac{\gamma_{1}}{\varepsilon}(t-s)}|B_{s}-\Phi_{s}^{\mu}|\mathrm{d}s\leq\frac{\mu^{\alpha}}{\sqrt{\varepsilon}}\frac{||A\sigma||_{\mathbb{H}_{1}}}{\gamma_{1}}, \quad\mbox{ for } \quad s\in[-T_{2},0).$$
$$\frac{||A\sigma||_{\mathbb{H}_{1}}}{\varepsilon\sqrt{\varepsilon}}e^{\frac{\rho}{\varepsilon}t}\int_{-\infty}^{-T_{2}}e^{-\frac{\gamma_{1}}{\varepsilon}(t-s)}|B_{s}-\Phi_{s}^{\mu}|\mathrm{d}s\leq C(\omega)\frac{\mu}{\sqrt{\varepsilon}}\frac{||A\sigma||_{\mathbb{H}_{1}}}{\beta},\quad\mbox{ for } \quad s\in(-\infty,-T_{2})$$
Thus 
$$I_{2}=o(\mu^{\alpha}),\quad\mbox{a.s. for every}\,\, \varepsilon.$$
Finally, we get 
$$||\hat{x}^{\mu,\varepsilon}(\theta_{t}\omega)-\hat{u}^{\varepsilon}(\theta_{t}\omega)||_{C_{-\frac{\rho}{\varepsilon}}^{1,-}}=o(\mu^{\alpha})\quad\mbox{a.s. for every}\,\, \varepsilon.$$
Based on the preceding derivation, we can obtain 
$$||h^{\mu,\varepsilon}(\omega,\xi)-h^{\varepsilon}(\omega,\xi)||_{\mathbb{H}_{1}}
\leq(\frac{K}{\gamma_{}-\rho}\frac{1}{1-\kappa_{1}(\varepsilon)}+1)||\hat{x}^{\mu,\varepsilon}(\theta_{s}\omega)-\hat{u}^{\varepsilon}(\theta_{s}\omega)||_{C_{-\frac{\rho}{\varepsilon}}^{1,-}}$$
That is $$||h^{\mu,\varepsilon}(\omega,\xi)-h^{\varepsilon}(\omega,\xi)||_{\mathbb{H}_{1}}=o(\mu^{\alpha}),\quad\mbox{a.s. for } \alpha\in(0,\frac{1}{2}).$$
Hence, Wong-Zakai random slow manifold approximates the original random slow manifold, as $\mu$ tends to zero.

Furthermore, we can obtain that this relation holds uniformly on interval $[\varepsilon_{1},\varepsilon_{2}]$ for any $\varepsilon_{1}$, $\varepsilon_{2}$ satisfying $0<\varepsilon_{1}<\varepsilon_{2}<\frac{\rho}{\gamma_{2}+\frac{1}{\frac{1}{K}-\frac{1}{\gamma_{1}-\rho}}}$.\\
In fact, inequality (\ref{ine2}) and the fact that $\kappa_{1}(\varepsilon)=\frac{K}{\gamma_{1}-\rho}+\frac{\varepsilon K}{\rho-\varepsilon\gamma_{2}}$ is a continuous function on interval $[\varepsilon_{1},\varepsilon_{2}]$,
and
$$I_{1}<\delta^{2}\frac{||\sigma||_{\mathbb{H}_{1}}}{\sqrt{\varepsilon}}+C\frac{||\sigma||_{\mathbb{H}_{1}}}{\sqrt{\varepsilon}}\mu^{\alpha}\triangleq m_{1}(\varepsilon)\mu^{\alpha},$$
$$I_{2}<C\frac{\mu^{\alpha}}{\sqrt{\varepsilon}}\frac{||A\sigma||_{\mathbb{H}_{1}}}{\gamma_{1}}+C\frac{\mu^{\alpha}}{\sqrt{\varepsilon}}\frac{||A\sigma||_{\mathbb{H}_{1}}}{\beta}\triangleq m_{2}(\varepsilon)\mu^{\alpha},$$
with a positive constant $C$, and we can take $\delta^{2}=C\mu^\alpha$. On the interval $[\varepsilon_{1},\varepsilon_{2}]$, the function $m_{1}(\varepsilon)$ is a continuous function of $\varepsilon$. On account of the extreme value theorem about continuous function on closed and bounded interval, this function must attain a maximal value on the interval $[\varepsilon_{1},\varepsilon_{2}]$. Hence $I_{1}=o(\mu^{\alpha})$ uniformly for $\varepsilon$ on the interval $[\varepsilon_{1},\varepsilon_{2}]$. Similarly, $I_{2}=o(\mu^{\alpha})$ uniformly for $\varepsilon$ on the interval $[\varepsilon_{1},\varepsilon_{2}]$. We obtain that
$$||\hat{x}^{\mu,\varepsilon}(\theta_{t}\omega)-\hat{u}^{\varepsilon}(\theta_{t}\omega)||_{C_{-\frac{\rho}{\varepsilon}}^{1,-}}=o(\mu^{\alpha}), \mbox{ a.s. }$$
uniformly on the interval $[\varepsilon_{1},\varepsilon_{2}]$.
\end{proof}

\begin{Remark}
It is not necessarily uniform on the interval $(0,\varepsilon_{1})$. We use an example to show the nonuniformity for $\varepsilon$ on the interval $(0,\varepsilon_{1})$. Notice that
\begin{eqnarray*}
\begin{aligned}
&||\hat{x}^{\mu,\varepsilon}(\theta_{t}\omega)-\hat{u}^{\varepsilon}(\theta_{t}\omega)||_{C_{-\frac{\rho}{\varepsilon}}}^{1,-}\\
=&\sup_{t\leq0}||e^{\frac{\rho}{\varepsilon}t}\frac{\sigma}{\sqrt{\varepsilon}}(\Phi_{t}^{\mu}-B_{t})+e^{\frac{\rho}{\varepsilon}t}\frac{1}{\varepsilon\sqrt{\varepsilon}}\int_{-\infty}^{t}e^{\frac{A}{\varepsilon}(t-s)}A\sigma(\Phi_{s}^{\mu}-B_{s})\mathrm{d}s||_{\mathbb{H}_{1}}.\\
\triangleq& \sup_{t\leq0} n(\varepsilon,\mu,t).\\
\end{aligned}
\end{eqnarray*}
Then $$||\hat{x}^{\mu,\varepsilon}(\theta_{t}\omega)-\hat{u}^{\varepsilon}(\theta_{t}\omega)||_{C_{-\frac{\rho}{\varepsilon}}}^{1,-}\geq n(\varepsilon,\mu,0),$$
with
\begin{eqnarray*}
\begin{aligned}
n(\varepsilon,\mu,0)=&||\frac{1}{\varepsilon\sqrt{\varepsilon}}\int_{-\infty}^{0}e^{-\frac{A}{\varepsilon}s}A\sigma(\Phi_{s}^{\mu}-B_{s})\mathrm{d}s||_{\mathbb{H}_{1}}\\
\end{aligned}
\end{eqnarray*}
We just need to show that $n(\varepsilon,\mu,0)\leq C\mu^{\alpha}$ doesn't hold uniformly for $\varepsilon\in(0,\varepsilon_{1})$.\\
By the stochastic Fubini's theorem \cite{P. E. Protter},
\begin{eqnarray*}
\begin{aligned}
\Phi_{t}^{\mu}(\omega)=&\frac{1}{\mu}\int_{0}^{t}\int_{-\infty}^{s}e^{-\frac{1}{\mu}(s-r)}\mathrm{d}B_{r}\mathrm{d}s\\
=&\frac{1}{\mu}\int_{-\infty}^{0}\int_{0}^{t}e^{-\frac{1}{\mu}(s-r)}\mathrm{d}s\mathrm{d}B_{r}+\frac{1}{\mu}\int_{0}^{t}\int_{r}^{t}e^{-\frac{1}{\mu}(s-r)}\mathrm{d}s\mathrm{d}B_{r}\\
=&\int_{-\infty}^{0}(e^{\frac{1}{\mu}r}-e^{-\frac{1}{\mu}(t-r)})\mathrm{d}B_{r}+\int_{0}^{t}(1-e^{-\frac{1}{\mu}(t-r)})\mathrm{d}B_{r}\\
=&\int_{-\infty}^{0}e^{\frac{1}{\mu}r}\mathrm{d}B_{r}-e^{-\frac{1}{\mu}t}\int_{-\infty}^{t}e^{\frac{1}{\mu}r}\mathrm{d}B_{r}+B_{t}.\\
\end{aligned}
\end{eqnarray*}
We get $$n(\varepsilon,\mu,0)=||\frac{1}{\varepsilon\sqrt{\varepsilon}}\int_{-\infty}^{0}e^{-\frac{A}{\varepsilon}s}A\sigma(\int_{-\infty}^{0}e^{\frac{1}{\mu}r}\mathrm{d}B_{r}-e^{-\frac{1}{\mu}s}\int_{-\infty}^{s}e^{\frac{1}{\mu}r}\mathrm{d}B_{r})\mathrm{d}s||_{\mathbb{H}_{1}}.$$
Denote
$$n_{1}(\varepsilon,\mu,0)=\frac{1}{\varepsilon\sqrt{\varepsilon}}\int_{-\infty}^{0}e^{-\frac{A}{\varepsilon}s}A\sigma\mathrm{d}s\int_{-\infty}^{0}e^{\frac{1}{\mu}r}\mathrm{d}B_{r},$$
$$n_{2}(\varepsilon,\mu,0)=\frac{1}{\varepsilon\sqrt{\varepsilon}}\int_{-\infty}^{0}e^{-\frac{A}{\varepsilon}s}A\sigma e^{-\frac{1}{\mu}s}\int_{-\infty}^{s}e^{\frac{1}{\mu}r}\mathrm{d}B_{r}\mathrm{d}s.$$
If $A=(-1)$, $\sigma=1$ then
\begin{eqnarray*}
\begin{aligned}
n_{1}(\varepsilon,\mu,0)
=\frac{1}{\sqrt{\varepsilon}}\int_{-\infty}^{0}e^{\frac{1}{\mu}r}\mathrm{d}B_{r}
=o(\varepsilon^{-1/2}).\\
\end{aligned}
\end{eqnarray*}
\begin{eqnarray*}
\begin{aligned}
n_{2}(\varepsilon,\mu,0)=\frac{1}{\varepsilon\sqrt{\varepsilon}}\frac{1}{\frac{1}{\varepsilon}-\frac{1}{\mu}}\int_{-\infty}^{0}(e^{\frac{1}{\mu}r}-e^{\frac{1}{\varepsilon}r})\mathrm{d}B_{r}
=o(\varepsilon^{-\frac{1}{2}}).
\end{aligned}
\end{eqnarray*}
In the last step, we use the fact that $\int_{-\infty}^{t}(e^{\frac{1}{\mu}r}-e^{\frac{1}{\varepsilon}r})\mathrm{d}B_{r}$ is a martingale with the quadratic variation $\langle \int_{-\infty}^{t}(e^{\frac{1}{\mu}r}-e^{\frac{1}{\varepsilon}r})\mathrm{d}B_{r} \rangle_{t}=\frac{\mu}{2}e^{\frac{2}{\mu}t}+\frac{\varepsilon}{2}e^{\frac{2}{\varepsilon}t}-\frac{2}{\frac{1}{\mu}+\frac{1}{\varepsilon}}e^{(\frac{1}{\mu}+\frac{1}{\varepsilon})t}$. The martingale $\int_{-\infty}^{t}(e^{\frac{1}{\mu}r}-e^{\frac{1}{\varepsilon}r})\mathrm{d}B_{r}$ can be expressed as a time changed Brownian motion $\cite{S. Al-azzawi, TaoJiang, L. Karatzas and S.E. Shreve}$ almost surely, that is
$$\int_{-\infty}^{t}(e^{\frac{1}{\mu}r}-e^{\frac{1}{\varepsilon}r})\mathrm{d}B_{r}=\tilde{B}_{\frac{\mu}{2}e^{\frac{2}{\mu}t}+\frac{\varepsilon}{2}e^{\frac{2}{\varepsilon}t}-\frac{2}{\frac{1}{\mu}+\frac{1}{\varepsilon}}e^{(\frac{1}{\mu}+\frac{1}{\varepsilon})t}} \qquad\mbox{ a.s. }$$
Take $t=0$ we can obtain
$$\int_{-\infty}^{0}(e^{\frac{1}{\mu}r}-e^{\frac{1}{\varepsilon}r})\mathrm{d}B_{r}=\tilde{B}_{\frac{\mu}{2}+\frac{\varepsilon}{2}-\frac{2}{\frac{1}{\mu}+\frac{1}{\varepsilon}}}\to{\tilde{B}}_{\frac{\mu}{2}} \quad\mbox{as } \varepsilon\to0 \qquad\mbox{ a.s. }$$
Notice that the Brownian motion $B_{t}$, $\tilde{B}_{t}$ and $\tilde{\tilde{B}}_{t}$ are different.
Hence $n(\varepsilon,\mu,0)=o(\varepsilon^{-\frac{1}{2}})$. If $A=(-1)$, $\sigma=1$ then the relation
$$||\hat{x}^{\mu,\varepsilon}(\theta_{t}\omega)-\hat{u}^{\varepsilon}(\theta_{t}\omega)||_{C_{-\frac{\rho}{\varepsilon}}^{1,-}}=o(\mu^{\alpha}), \mbox{ a.s. }$$
doesn't hold uniformly for $\varepsilon$ on interval $(0,\varepsilon_{1})$.
\end{Remark}

\medskip

Ultimately, we get our main results about Wong-Zakai approximation of the original random slow manifold, the exponential tracking property of Wong-Zakai random slow manifold about orbits of the original system, and furthermore the Wong-Zakai reduction of the original system by using Wong-Zakai random slow manifold.

\bigskip

\noindent{\bf Proof of Theorem \ref{thm:intetrack}}

\medskip
\begin{proof}
According to the theorem about exponential tracking property of the original random slow manifold in \cite{HongboFuSlowmf}, there exists an initial point $\bar{\zeta}_{o}$ on the original random slow manifold $\mathcal{M}^{\varepsilon}(\omega)$ such that
$$||w^{\varepsilon}(t,\omega,\zeta)-w^{\varepsilon}(t,\omega,\bar{\zeta}_{o})||_{\mathbb{H}}\leq C_{1}e^{-C_{2}t}||\zeta-\bar{\zeta}_{o}||_{\mathbb{H}},\quad\mbox{for}\quad t\geq0,$$
with $C_{1}=\frac{1}{1-K(\frac{1}{\gamma_{1}-\rho}+\frac{\varepsilon}{\rho-\varepsilon\gamma_{2}})}$ and $C_{2}=\frac{\rho}{\varepsilon}$.\\

\medskip
\noindent For this $\bar{\zeta}_{o}=(h^{\varepsilon}(\omega,\bar{\xi}),\bar{\xi})$, point $\bar{\zeta}=(h^{\mu,\varepsilon}(\omega,\bar{\xi}),\bar{\xi})$ is on Wong-Zakai random slow manifold $\mathcal{M}^{\mu,\varepsilon}(\omega)$.
Based on systems (\ref{eq:orgSDE}), (\ref{eq:newRDE}), Proposition \ref{OU-RSM}, Theorem \ref{thm:appslmf}, and the invariant property of random slow manifold,
\begin{eqnarray*}
\begin{aligned}
\dot{v}^{\varepsilon}(t)-\dot{y}^{\mu,\varepsilon}(t)=&B(v^{\varepsilon}(t)-y^{\mu,\varepsilon}(t))
+g(h^{\varepsilon}(\theta_{t}\omega,v^{\varepsilon}(t)),v^{\varepsilon}(t))-g(h^{\mu,\varepsilon}(\theta_{t}\omega,y^{\mu,\varepsilon}(t)),y^{\mu,\varepsilon}(t))),
\end{aligned}
\end{eqnarray*}
with initial value ${v}^{\varepsilon}(0)-{y}^{\mu,\varepsilon}(0)=\bar{\xi}-\bar{\xi}=0$.\\
According to the $C_{0}-$semigroup property of operator $B$, there exist positive constants $M$ and $\gamma$ such that the operator norm of $e^{Bt}$ is exponential bounded, that is
$$||e^{Bt}||<Me^{\gamma t}, \mbox{ for } t>0.$$
Due to Theorem \ref{thm:appslmf}, the variation of constant formula, and the semigroup property of operator $B$, we conclude
\begin{eqnarray*}
\begin{aligned}
&||v^{\varepsilon}(t)-y^{\mu,\varepsilon}(t)||_{\mathbb{H}_{2}}\\
\leq&||\int_{0}^{t}e^{B(t-s)}(g(h^{\varepsilon}(\theta_{s}\omega,v^{\varepsilon}(s)),v^{\varepsilon}(s))-g(h^{\mu,\varepsilon}(\theta_{s}\omega,y^{\mu,\varepsilon}(s)),y^{\mu,\varepsilon}(s)))\mathrm{d}s||_{\mathbb{H}_{2}}\\
\leq&\int_{0}^{t}Me^{\gamma(t-s)}K(||h^{\varepsilon}(\theta_{s}\omega,v^{\varepsilon}(s))-h^{\varepsilon}(\theta_{s}\omega,y^{\mu,\varepsilon}(s))||_{\mathbb{H}_{1}}\\
&\quad+||h^{\varepsilon}(\theta_{s}\omega,y^{\mu,\varepsilon}(s))-h^{\mu,\varepsilon}(\theta_{s}\omega,y^{\mu,\varepsilon}(s))||_{\mathbb{H}_{1}}+||v^{\varepsilon}(s)-y^{\mu,\varepsilon}(s)||_{\mathbb{H}_{2}})\mathrm{d}s\\
\leq&\int_{0}^{t}Me^{\gamma(t-s)}K(Lip h^{\varepsilon} +1)||v^{\varepsilon}(s)-y^{\mu,\varepsilon}(s)||_{\mathbb{H}_{2}}\mathrm{d}s+o(\mu^{\alpha}).\\
\end{aligned}
\end{eqnarray*}
That is
$$e^{-\gamma t}||v^{\varepsilon}(t)-y^{\mu,\varepsilon}(t)||_{\mathbb{H}_{2}}\leq o(\mu^{\alpha})+\int_{0}^{t}Me^{-\gamma s}K(Lip h^{\varepsilon} +1)||v^{\varepsilon}(s)-y^{\mu,\varepsilon}(s)||_{\mathbb{H}_{2}}\mathrm{d}s.$$
By the integral form of Gronwall inequality, for any fixed $t\geq0$,
\begin{eqnarray*}
\begin{aligned}
||v^{\varepsilon}(t)-y^{\mu,\varepsilon}(t)||_{\mathbb{H}_{2}}
\leq&o(\mu^{\alpha})e^{\gamma t}\int_{0}^{t}e^{\int_{s}^{t}K(1+Lip h^{\varepsilon})\mathrm{d}\tau}\mathrm{d}s\\
\leq&o(\mu^{\alpha})\frac{1}{K(1+Lip h^{\varepsilon})}(e^{K(1+Lip h^{\varepsilon})t}-1)e^{\gamma t}\\
=&o(\mu^{\alpha}).
\end{aligned}
\end{eqnarray*}
Based on Theorem \ref{thm:appslmf}, for almost sure $\omega\in\Omega$ and every fixed $t\geq0$,
$$||\bar{\zeta}-\bar{\zeta}_{o}||_{\mathbb{H}}=||h^{\mu,\varepsilon}(\omega,\bar{\xi})-h^{\varepsilon}(\omega,\bar{\xi})||_{\mathbb{H}_{1}}+||\bar{\xi}-\bar{\xi}||_{\mathbb{H}_{1}}=o(\mu^{\alpha}).$$
Hence for almost sure $\omega\in\Omega$ and every $t\geq0$,
\begin{eqnarray*}
\begin{aligned}
&||w^{\varepsilon}(t,\omega,\zeta)-z^{\mu,\varepsilon}(t,\omega,\bar{\zeta})||_{\mathbb{H}}\\
\leq&||w^{\varepsilon}(t,\omega,\zeta)-w^{\varepsilon}(t,\omega,\bar{\zeta}_{o})||_{\mathbb{H}}+||w^{\varepsilon}(t,\omega,\bar{\zeta}_{o})-z^{\mu,\varepsilon}(t,\omega,\bar{\zeta})||_{\mathbb{H}}\\
\leq&C_{1}e^{-C_{2}t}||\zeta-\bar{\zeta}_{o}||_{\mathbb{H}}+||h^{\varepsilon}(\theta_{t}\omega,v^{\varepsilon}(t,\omega,\bar{\zeta}_{o}))-h^{\varepsilon}(\theta_{t}\omega,y^{\mu,\epsilon}(t,\omega,\bar{\zeta}))||_{\mathbb{H}_{1}}\\
&+||h^{\varepsilon}(\theta_{t}\omega,y^{\mu,\varepsilon}(t,\omega,\bar{\zeta}))-h^{\mu,\varepsilon}(\theta_{t}\omega,y^{\mu,\varepsilon}(t,\omega,\bar{\zeta}))||_{\mathbb{H}_{1}}
+||v^{\varepsilon}(t,\omega,\bar{\zeta}_{o})-y^{\mu,\varepsilon}(t,\omega,\bar{\zeta})||_{\mathbb{H}_{2}}\\
\leq&C_{1}e^{-C_{2}t}(||\zeta-\bar{\zeta}||_{\mathbb{H}}+||\bar{\zeta}-\bar{\zeta}_{o}||_{\mathbb{H}})+(1+Lip h^{\varepsilon})||v^{\varepsilon}(t,\omega,\bar{\zeta}_{o})-y^{\mu,\varepsilon}(t,\omega,\bar{\zeta})||_{\mathbb{H}_{2}}+o(\mu^{\alpha})\\
=&C_{1}e^{-C_{2}t}||\zeta-\bar{\zeta}||_{\mathbb{H}}+o(\mu^{\alpha}).
\end{aligned}
\end{eqnarray*}
The proof is complete.
\end{proof}

Theorem \ref{thm:intetrack} provides us a theoretical basis to project the dynamical behavior of the original system on Wong-Zakai random slow manifold. Thus we get the following reduced system
\begin{eqnarray*}
\left\{\begin{array}{l}
\begin{aligned}
\dot{\tilde{v}}^{\varepsilon}&=B\tilde{v}^{\varepsilon}+g(\tilde{u}^{\varepsilon},\tilde{v}^{\varepsilon}),\\
\tilde{u}^{\varepsilon}&=h^{\mu,\varepsilon}(\theta_{t}\omega,\tilde{v}^{\varepsilon}).
\end{aligned}
\end{array}
\right.
\end{eqnarray*}
Corollary \ref{thm:intereduce} is established.
This system can be regarded as a deterministic system point wisely. Through the qualitative or quantitative characteristics of the lower dimensional deterministic system point wisely, we can detect the behavior of the original stochastic dynamical system.

\section{Application to  parameter estimation}\label{sec:estimate}
In this section, the Wong-Zakai reduction system is used to estimate a  parameter of the original system.
Consider a slow-fast stochastic system
\begin{eqnarray}\label{eq:orgSDE2}
\left\{\begin{array}{l}
\begin{aligned}
\dot{u}^{\varepsilon}=\frac{1}{\varepsilon}Au^{\varepsilon}+\frac{1}{\varepsilon}f(u^{\varepsilon},v^{\varepsilon})+\frac{\sigma}{\sqrt\varepsilon}\dot{B}_{t}(\omega), \qquad  \hbox{in}\,\,&\mathbb{H}_{1},\\
\dot{v}^{\varepsilon}=Bv^{\varepsilon}+g(u^{\varepsilon},v^{\varepsilon},a),  ~~~~~~~~~~~~~~~~~\qquad\hbox{in}\,\,&\mathbb{H}_{2}.
\end{aligned}
\end{array}
\right.
\end{eqnarray}
Here the slow subsystem    contains a parameter $a$, taking value in a closed interval $\Lambda$ in $\mathbb{R}$.
It has been proved in \cite{Ren jian} that when only slow variable is observable, we can get a good estimator using the
reduced system
\begin{eqnarray}\label{estsl}
\dot{v}^{\varepsilon}(t)=Bv^{\varepsilon}(t)+g(h^{\varepsilon}(\theta_{t}\omega,v^{\varepsilon}(t)),v^{\varepsilon}(t),a), \quad\hbox{in}\quad \mathbb{H}_{2},
\end{eqnarray}
based on exponential tracking property of the original random slow manifold $\mathcal{M}^{\varepsilon}$.  This parameter estimator,  via a stochastic Nelder-Mead simulation method,  is a good approximation to that using the original system directly with  both fast and slow variables observable. This reduces the amount of information needed from observation and saves the computational cost.  But this reduction still needs a stochastic simulation method, which is more difficult than deterministic simulation method. Since we   get the intersystem exponential tracking property of the Wong-Zakai random slow manifold $\mathcal{M}^{\mu,\varepsilon}$ to the original system (\ref{eq:orgSDE2}), we now  devise a parameter estimator using the reduced system
\begin{eqnarray}\label{estousl}
\dot{v}^{\varepsilon}(t)=Bv^{\varepsilon}(t)+g(h^{\mu,\varepsilon}(\theta_{t}\omega,v^{\varepsilon}(t)),v^{\varepsilon}(t),a), \quad\hbox{in}\quad \mathbb{H}_{2},
\end{eqnarray}
based on the deterministic simulation method. It turns out this  is an accurate estimator, comparing with the one using the original system directly.

Given the initial state  $\zeta=(\eta,\xi)$, let $v^{\varepsilon}_{ob}(t)$ be the observation of slow variable $v^{\varepsilon}(t)$ generated by a true system parameter value $a$.   We do    not use the observation of the fast variable $u^{\varepsilon}(t)$, but denote it as $u^{\varepsilon}_{ob}(t)$ just formally. Take initial state  to be $\xi$ in the two reduced systems. We denote the parameter in the Wong-Zakai reduced system (\ref{estousl}) as $a^{E}_{WS}$, with the  corresponding solution as $v^{\varepsilon}_{WS}(t)$. We denote the parameter in reduced system (\ref{estsl}) as $a_{S}^{E}$,  with the corresponding solution as $v^{\varepsilon}_{S}(t)$.
Define the square of the observation error as the objective function
$$F^{WS}(a^{E}_{WS})=\mathbb{E}\int_{0}^{T}||v^{\varepsilon}_{WS}(t)-v^{\varepsilon}_{ob}(t)||^{2}_{\mathbb{H}_{2}}\mathrm{d}s,$$
and take the minimizer $a^{E}_{WS}$ of function $F^{WS}(\cdot)$ as the estimation of the true parameter value $a$. Note that the square root of the minimum value of this objective function has error $o(\varepsilon)+o(\mu^{\alpha})$ with any $\alpha\in(0,\frac{1}{2})$ compared to zero due to the Wong-Zakai slow reduction according to Corollary \ref{thm:intereduce}.
We can verify that the difference $|a^{E}_{WS}-a|$ can be controlled by objective function and the scale parameter, that is $(F^{WS}(a^{E}_{WS}))^{\frac{1}{2}}$ and $o(\varepsilon)$, as follows
\begin{eqnarray*}
\begin{aligned}
&|a^{E}_{WS}-a|\cdot\mathbb{E}||\int_{0}^{t^{*}}e^{-Bt}\bigtriangledown_{a}g(h^{\mu,\varepsilon}(\theta_{t}\omega,v^{\varepsilon}_{WS}(t)),v^{\varepsilon}_{WS}(t),a^{'})\mathrm{d}t||_{\mathbb{H}_{2}}\\
<&\frac{\varepsilon}{\gamma_{1}-K}||\eta-h^{\mu,\varepsilon}(\xi)||_{\mathbb{H}_{1}}+\big[(1+K)e^{\gamma_{2}T}+\frac{K}{\gamma_{1}-K}\big](T\cdot F^{WS}(a^{E}_{WS}))^{\frac{1}{2}},
\end{aligned}
\end{eqnarray*}
where $T>0$, $t^{*}\in(0,T)$, $a^{'}\in(a^{E}_{WS},a)$ (or $a^{'}\in(a,a^{E}_{WS})$).\\
Denote
$$G(a,a^{E}_{WS})\triangleq||\int_{0}^{t^{*}}e^{-Bt}\bigtriangledown_{a}g(h^{\mu,\varepsilon}(\theta_{t}\omega,v^{\varepsilon}_{WS}(t)),v^{\varepsilon}_{WS}(t,a^{'})\mathrm{d}t||_{\mathbb{H}_{2}}.$$
Assume that the function $G(a,a^{E}_{WS})$ satisfies $0<G(a,a^{E}_{WS})<\infty$, and one component of function $\mathbb{E}g(u,v,a)$ is strictly monotonic with respect to $a\in\Lambda$ constrained to Wong-Zakai random slow manifold $\mathcal{M}^{\mu,\varepsilon}$. Then with the same discussion as \cite{Ren jian}, $|a^{E}_{WS}-a|$ is controlled by the Wong-Zakai slow reduction $||\eta-h^{\mu,\varepsilon}(\xi)||_{\mathbb{H}_{1}}$ to order $o(\varepsilon)$ and the objective function
$F^{WS}(a^{E}_{WS})$ to order $o(\mu^{\alpha})$ with any $\alpha\in(0,\frac{1}{2})$.

We can expand $h^{\mu,\varepsilon}(\theta_{t}\omega,v^{\varepsilon}(t))$ with respect to small $\varepsilon$ to order $o(\varepsilon^{2})$ like \cite{Ren jian 2} as follows
$$\hat{h}^{\mu,\varepsilon}(\theta_{t}\omega,v^{\varepsilon}(t))=h^{\mu,\varepsilon}_{1}+\varepsilon (h^{\mu,\varepsilon}_{2}+o(\mu^{\frac{1}{2}}))+o(\varepsilon^{2}).$$
Then system (\ref{estousl}) can be replaced by system
\begin{eqnarray}\label{expestousl}
\dot{v}^{\varepsilon}(t)=Bv^{\varepsilon}(t)+g(\hat{h}^{\mu,\varepsilon}(\theta_{t}\omega,v^{\varepsilon}(t)),v^{\varepsilon}(t),a), \quad\hbox{in}\quad \mathbb{H}_{2},
\end{eqnarray}
for the estimation.
Since
\begin{eqnarray*}
\begin{aligned}
&||v^{\varepsilon}_{ob}(t)-h^{\mu,\varepsilon}(\theta_{t}\omega,v^{\varepsilon}(t))||_{\mathbb{H}_{1}}\\
\leq&||v^{\varepsilon}_{ob}(t)-\hat{h}^{\mu,\varepsilon}(\theta_{t}\omega,v^{\varepsilon}(t))||_{\mathbb{H}_{1}}+||\hat{h}^{\mu,\varepsilon}(\theta_{t}\omega,v^{\varepsilon}(t))-h^{\mu,\varepsilon}(\theta_{t}\omega,v^{\varepsilon}(t))||_{\mathbb{H}_{1}},
\end{aligned}
\end{eqnarray*}
the error of estimator using system (\ref{expestousl}) is thus  bounded by $(F^{WS}(a^{E}_{WS}))^{\frac{1}{2}}$ and $o(\varepsilon)$.\\

We use the following example to illustrate our method for  parameter estimation.
\begin{Example}\label{example}
The SDE with a unknown parameter $a$ driven by Brownian motion is
\begin{eqnarray}\label{exorgSDE}
\left\{\begin{array}{l}
\dot{u}^{\varepsilon}=-\frac{1}{\varepsilon}u^{\varepsilon}+\frac{1}{600\varepsilon}(v^{\varepsilon})^{2}+\frac{0.1}{\sqrt{\varepsilon}}\dot{B}_{t},\,\hbox{in } \mathbb{R},\\
\dot{v}^{\varepsilon}=0.001v^{\varepsilon}-au^{\varepsilon}v^{\varepsilon},\,\hbox{in } \mathbb{R}.
\end{array}
\right.
\end{eqnarray}
\end{Example}
Its approximation system is
\begin{eqnarray}\label{exnewSDE}
\left\{\begin{array}{l}
\dot{x}^{\mu,\varepsilon}=-\frac{1}{\varepsilon}x^{\mu,\varepsilon}+\frac{1}{600\varepsilon}(y^{\mu,\varepsilon})^{2}+\frac{0.1}{\sqrt{\varepsilon}}\dot{\Phi}^{\mu}_{t},\,\hbox{in } \mathbb{R},\\
\dot{y}^{\mu,\varepsilon}=0.001y^{\mu,\varepsilon}-ax^{\mu,\varepsilon}y^{\mu,\varepsilon},\,\hbox{in } \mathbb{R}.
\end{array}
\right.
\end{eqnarray}
We compare the random slow manifolds of the two systems to interpret the previous procedure. In the following calculation we take the system parameter $a=0.1$.
In accordance with the former sections, $\mathbb{H}_{1}=\mathbb{H}_{2}=\mathbb{R}$, $\mathbb{H}=\mathbb{R}^{2}$ with $||\cdot||_{1}=||\cdot||_{2}=|\cdot|$,  $||\cdot||=||\cdot||_{1}+||\cdot||_{2}$ and
similar to $\cite{Ren jian,XingyeKan}$, there exist a random absobing set in both systems respectively. We can cut off the nonlinear functions to satisfy the global Lipschitz condition $(A2)$ without changing their long time behavior. Then there exist the Lipschitz random slow manifolds $\mathcal{M}^{\varepsilon}$ and $\mathcal{M}^{\mu,\varepsilon}$ of system $(\ref{exorgSDE})$ and $(\ref{exnewSDE})$ respectively.\\
System
$$\dot{\hat{u}}^{\varepsilon}=-\frac{1}{\varepsilon}\hat{u}^{\varepsilon}+\frac{0.1}{\sqrt{\varepsilon}}\dot{B}_{t}$$
has a stationary solution
$$\hat{u}^{\varepsilon}(\theta_{t}\omega)=\frac{0.1}{\sqrt{\varepsilon}}\int_{-\infty}^{t}e^{-\frac{1}{\varepsilon}(t-s)}\mathrm{d}B_{s},$$
through the random fixed point
$$\hat{u}^{\varepsilon}(\omega)=\frac{0.1}{\sqrt{\varepsilon}}\int_{-\infty}^{0}e^{\frac{1}{\varepsilon}s}\mathrm{d}B_{s}.$$
Through the random transformation $(\ref{orgradtrs})$,
system $(\ref{exorgSDE})$ becomes to
\begin{eqnarray}\label{exorgRDS}
\left\{\begin{array}{l}
\mathrm{d}U^{\varepsilon}(t)=-\frac{1}{\varepsilon}U^{\varepsilon}(t)\mathrm{d}t+\frac{1}{600\varepsilon} (V^{\varepsilon}(t))^{2}\mathrm{d}t,\\
\mathrm{d}V^{\varepsilon}(t)=0.001V^{\varepsilon}(t)\mathrm{d}t-0.1(U^{\varepsilon}(t)+\hat{u}^{\varepsilon}(\theta_{t}\omega))V^{\varepsilon}(t)\mathrm{d}t.
\end{array}
\right.
\end{eqnarray}
According to $\cite{HongboFuSlowmf}$, its random slow manifold is
$$M^{\varepsilon}(\omega)=\{(H^{\varepsilon}(\omega,\xi),\xi):\xi\in \mathbb{R}\},$$
where
$$H^{\varepsilon}(\omega,\xi)=\frac{1}{\varepsilon}\int_{-\infty}^{0}e^{\frac{1}{\varepsilon}s}\frac{(V^{\varepsilon}(s))^{2}}{600}\mathrm{d}s.$$
Furthermore, the random slow manifold of system $(\ref{exorgSDE})$ is
\begin{eqnarray*}
\begin{aligned}
\mathcal{M}^{\varepsilon}(\omega)=&M^{\varepsilon}(\omega)+(\hat{u}^{\varepsilon}(\omega),0)
=\{(h^{\varepsilon}(\omega,\xi)):\xi\in\mathbb{R}\},
\end{aligned}
\end{eqnarray*}
where
$$h^{\varepsilon}(\omega,\xi)=H^{\varepsilon}(\omega,\xi)+\hat{u}^{\varepsilon}(\omega).$$
Through the same procedure, system
$$\dot{\hat{x}}^{\mu,\varepsilon}=-\frac{1}{\varepsilon}\hat{x}^{\mu,\varepsilon}+\frac{0.1}{\sqrt{\varepsilon}}\dot{\Phi}^{\mu}_{t},$$
has a stationary solution
$$\hat{x}^{\mu,\varepsilon}(\theta_{t}\omega)=\frac{0.1}{\sqrt{\varepsilon}}\int_{-\infty}^{t}e^{-\frac{1}{\varepsilon}(t-s)}\mathrm{d}\Phi^{\mu}_{s},$$
through random variable
$$\hat{x}^{\mu,\varepsilon}(\omega)=\frac{0.1}{\sqrt{\varepsilon}}\int_{-\infty}^{0}e^{\frac{1}{\varepsilon}s}\mathrm{d}\Phi^{\mu}_{s}.$$
Through the random transformation $(\ref{newradtrs})$
system $(\ref{exnewSDE})$ becomes to
\begin{eqnarray}\label{exnewRDS}
\left\{\begin{array}{l}
\mathrm{d}X^{\mu,\varepsilon}(t)=-\frac{1}{\varepsilon}X^{\mu,\varepsilon}(t)\mathrm{d}t+\frac{1}{600\varepsilon} (Y^{\mu,\varepsilon}(t))^{2}\mathrm{d}t, \\
\mathrm{d}Y^{\mu,\varepsilon}(t)=0.001Y(t)\mathrm{d}t-0.1(X^{\mu,\varepsilon}(t)+\hat{x}^{\mu,\varepsilon}(\theta_{t}\omega))Y^{\mu,\varepsilon}(t)\mathrm{d}t.
\end{array}
\right.
\end{eqnarray}
According to Proposition $\ref{OU-RSM}$, its random slow manifold is
$$M^{\mu,\varepsilon}(\omega)=\{(H^{\mu,\varepsilon}(\omega,\xi),\xi):\xi\in\mathbb{R}\},$$
where
$$H^{\mu,\varepsilon}(\omega,\xi)=\frac{1}{\varepsilon}\int_{-\infty}^{0}e^{\frac{1}{\varepsilon}s}\frac{(Y^{\mu,\varepsilon}(s))^{2}}{600}\mathrm{d}s.$$
Furthermore, the random slow manifold of system $(\ref{exnewSDE})$ is
\begin{eqnarray*}
\begin{aligned}
\mathcal{M}^{\mu,\varepsilon}(\omega)=&M^{\mu,\varepsilon}(\omega)+(\hat{x}^{\mu,\varepsilon}(\omega),0)
=\{(h^{\mu,\varepsilon}(\omega,\xi),\xi):\xi\in\mathbb{R}\},
\end{aligned}
\end{eqnarray*}
where
$$h^{\mu,\varepsilon}(\omega,\xi)=H^{\mu,\varepsilon}(\omega,\xi)+\hat{x}^{\mu,\varepsilon}(\omega).$$
According to Theorem $\ref{thm:appslmf}$, for $\xi\in\mathbb{R}$
$$h^{\mu,\varepsilon}(\omega,\xi)\to h^{\varepsilon}(\omega,\xi)\quad\hbox{a.s. as }\quad\mu\to0.$$
That is the random slow manifold of system $(\ref{exnewSDE})$ approximate that of system $(\ref{exorgSDE})$ almost surely as Wong-Zakai approximation parameter $\mu\to0$.

Based on $\cite{Ren jian, Ren jian 2}$, we can plot the graph of slow manifolds $\mathcal{M}^{\varepsilon}(\omega)$ and $\mathcal{M}^{\mu,\varepsilon}(\omega)$ approximately to order o($\varepsilon^{2}$).
Let $t=\tau\varepsilon$
$$U(t)=U(\tau\varepsilon)=U_{0}(\tau)+\varepsilon U_{1}(\tau)+\varepsilon^{2}U_{2}(\tau)+\cdots$$
with initial value
\begin{eqnarray*}
\left\{\begin{array}{l}
U(0)=H(\xi,\omega)=H^{0}+\varepsilon H^{1}+\cdots\\
V(0)=\xi.
\end{array}
\right.
\end{eqnarray*}
$$H^{(0)}=\int_{-\infty}^{0}e^{-As}f(U_{0}+\hat{u}(\theta_{\varepsilon s}\omega),\xi)\mathrm{d}s=\frac{\xi^{2}}{600}.$$
$$U^{(0)}(\tau)=e^{A\tau}h^{(0)}+\int_{0}^{\tau}e^{-A(s-\tau)}f(U_{0}(s)+\hat{u}(\theta_{s\varepsilon}\omega),\xi)\mathrm{d}s=\frac{\xi^{2}}{600}.$$
By stochastic Fubini's theorem \cite{P. E. Protter}, we have
\begin{eqnarray*}
\begin{aligned}
H^{(1)}=&\int_{-\infty}^{0}e^{-As}\bigg\{f_{v}(U_{0}(s)+\hat{u}(\theta_{\varepsilon s}\omega),\xi)[Bs\xi+\int_{0}^{s}g(U_{0}(r)+\hat{u}(\theta_{\varepsilon r}\omega),\xi)\mathrm{d}r]\\
&\quad+f_{u}(U_{0}+\hat{u}(\theta_{\varepsilon r}\omega),\xi)U_{1}(s)\bigg\}\mathrm{d}s\\
=&-\frac{0.001\xi^{2}}{300}+\frac{0.1\xi^{4}}{180000}-\frac{(0.1)(0.1)\xi^{2}}{300}\int_{-\infty}^{0}se^{s}\mathrm{d}B_{s}.
\end{aligned}
\end{eqnarray*}

The slow manifold of system ($\ref{exorgRDS}$) is
\begin{eqnarray*}
\begin{aligned}
H^{\varepsilon}(\omega,\xi)=&H^{\varepsilon(0)}+\varepsilon H^{\varepsilon(1)}+o(\varepsilon^{2})\\
=&\frac{\xi^{2}}{600}+\varepsilon \bigg[-\frac{0.001\xi^{2}}{300}+\frac{0.1\xi^{4}}{180000}-\frac{0.01\xi^{2}}{300}\int_{-\infty}^{0}se^{s}\mathrm{d}B_{s}\bigg]+o(\varepsilon^{2}).\\
\end{aligned}
\end{eqnarray*}
The slow manifold of system ($\ref{exorgSDE}$) is
\begin{eqnarray}\label{exsmf1}
\qquad\begin{aligned}
&h^{\varepsilon}(\omega,\xi)\\
=&\frac{\xi^{2}}{600}+\varepsilon \bigg[-\frac{0.001\xi^{2}}{300}+\frac{0.1\xi^{4}}{180000}-\frac{0.01\xi^{2}}{300}\int_{-\infty}^{0}se^{s}\mathrm{d}B_{s}(\omega)\bigg]+o(\varepsilon^{2})+\frac{\sigma}{\sqrt{\varepsilon}}\int_{-\infty}^{0}e^{\frac{s}{\varepsilon}}\mathrm{d}B_{s}(\omega).\\
\end{aligned}
\end{eqnarray}
By (\ref{zakai}) and with the same procedure, we can infer that the slow manifold of system ($\ref{exnewSDE}$) is
\begin{eqnarray}\label{exsmf2}
\qquad\,\,\,\begin{aligned}
&h^{\mu,\varepsilon}(\omega,\xi)\\
=&\frac{\xi^{2}}{600}+\varepsilon \bigg[-\frac{0.001\xi^{2}}{300}+\frac{0.1\xi^{4}}{180000}-\frac{0.01\xi^{2}}{300}\int_{-\infty}^{0}se^{s}z_{s}(\omega)\mathrm{d}s\bigg]+o(\varepsilon^{2})+\frac{\sigma}{\sqrt{\varepsilon}}\int_{-\infty}^{0}e^{\frac{s}{\varepsilon}}z_{s}(\omega)\mathrm{d}s.
\end{aligned}
\end{eqnarray}

We omit the higher order term of $\varepsilon$ to plot the slow manifolds $h^{\varepsilon}(\omega,\xi)$ in (\ref{exsmf1})  and $h^{\mu,\varepsilon}(\omega,\xi)$ in (\ref{exsmf2}) in Figure \ref{figsm}. The random slow manifold of (\ref{exsmf1}) remains the same in figure \ref{figsm} to exhibit the approximation procedure. Figure $\ref{figsm}$ shows that the slow manifold of the original system ($\ref{exorgSDE}$) will be approximated by that of the Wong-Zakai system ($\ref{exnewSDE}$) when parameter $\mu$ goes to zero and $\varepsilon$ sufficiently small.

\begin{figure}[H]
\centering
\includegraphics[width=2in]{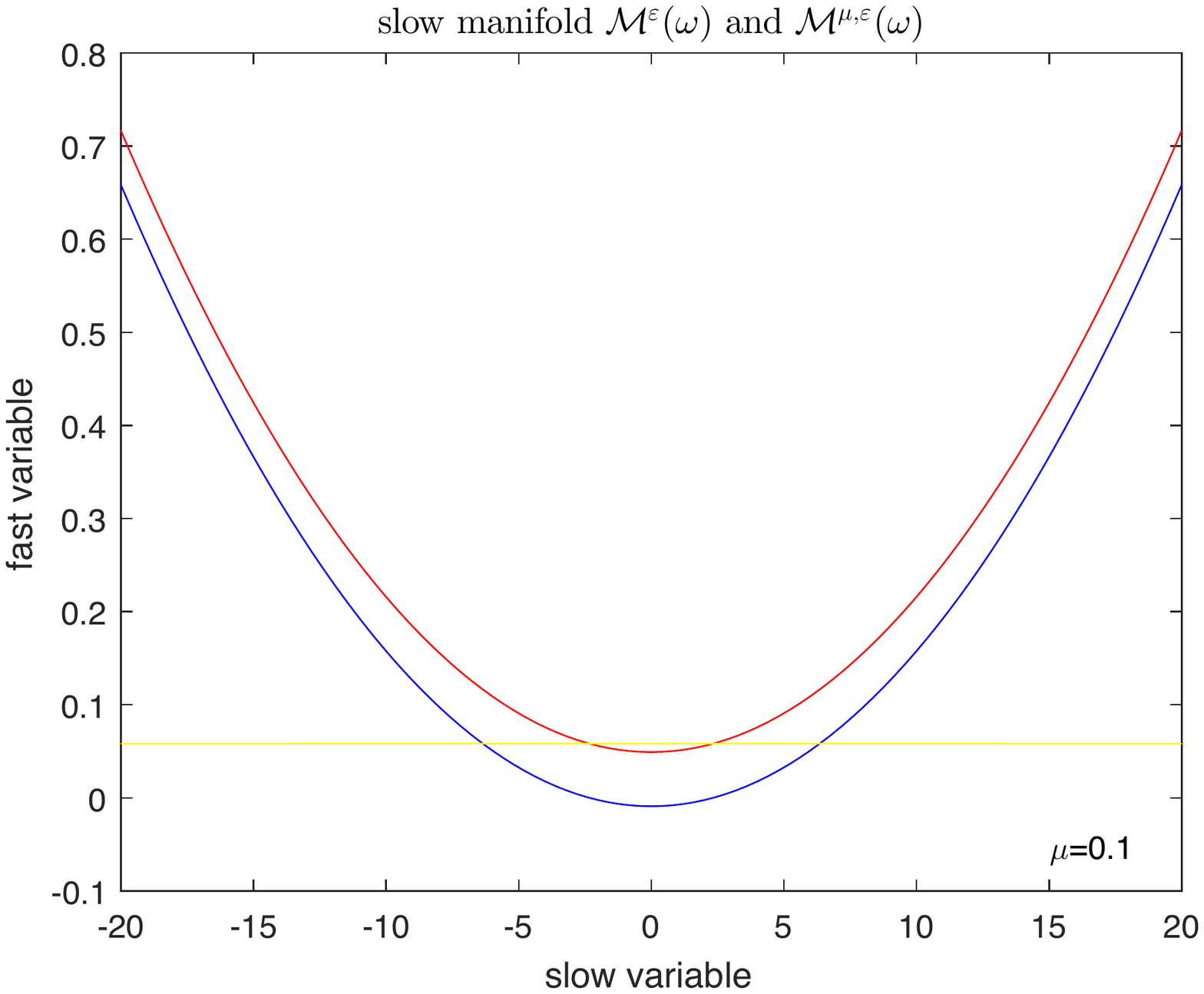}
\includegraphics[width=2in]{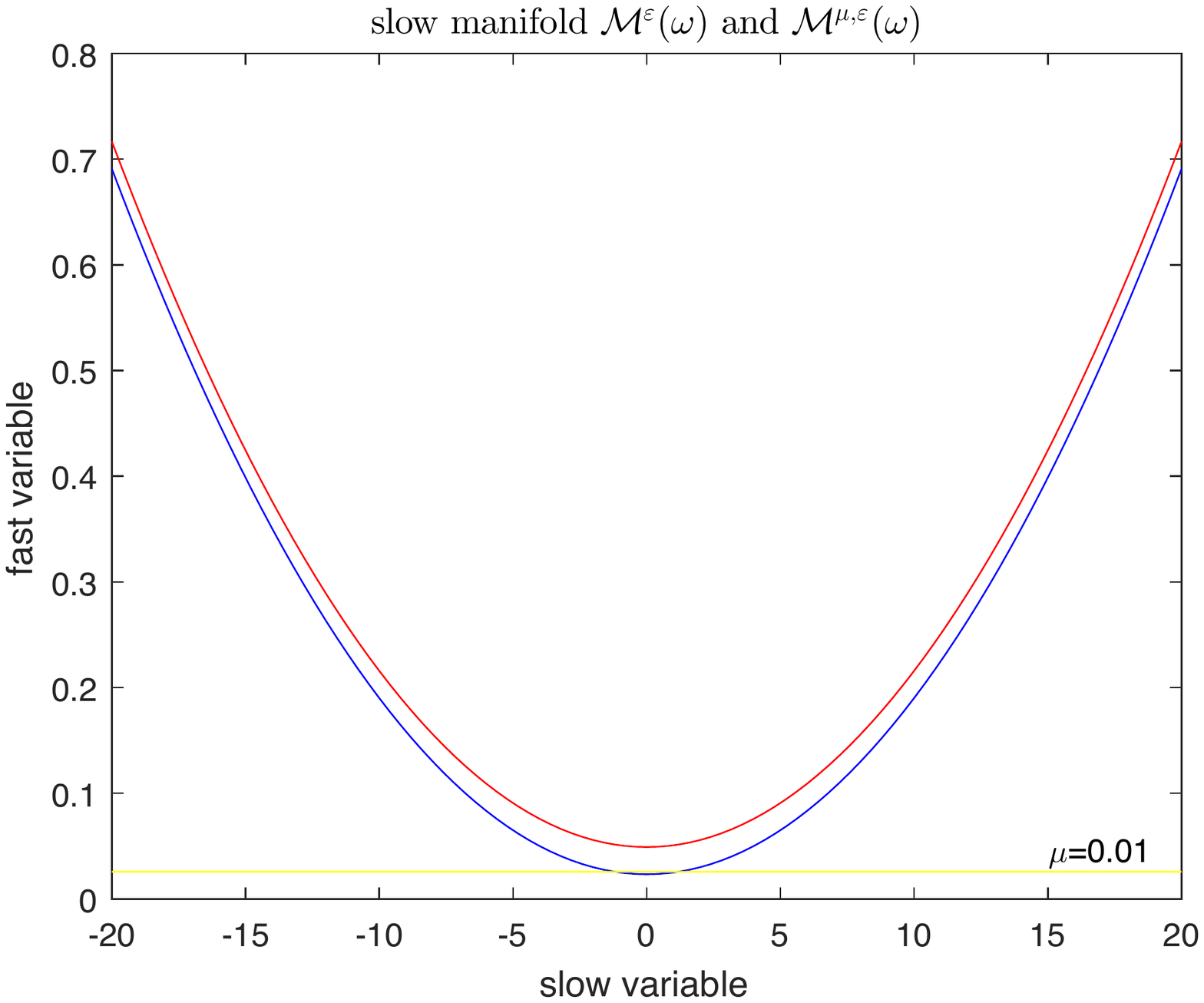}
\includegraphics[width=2in]{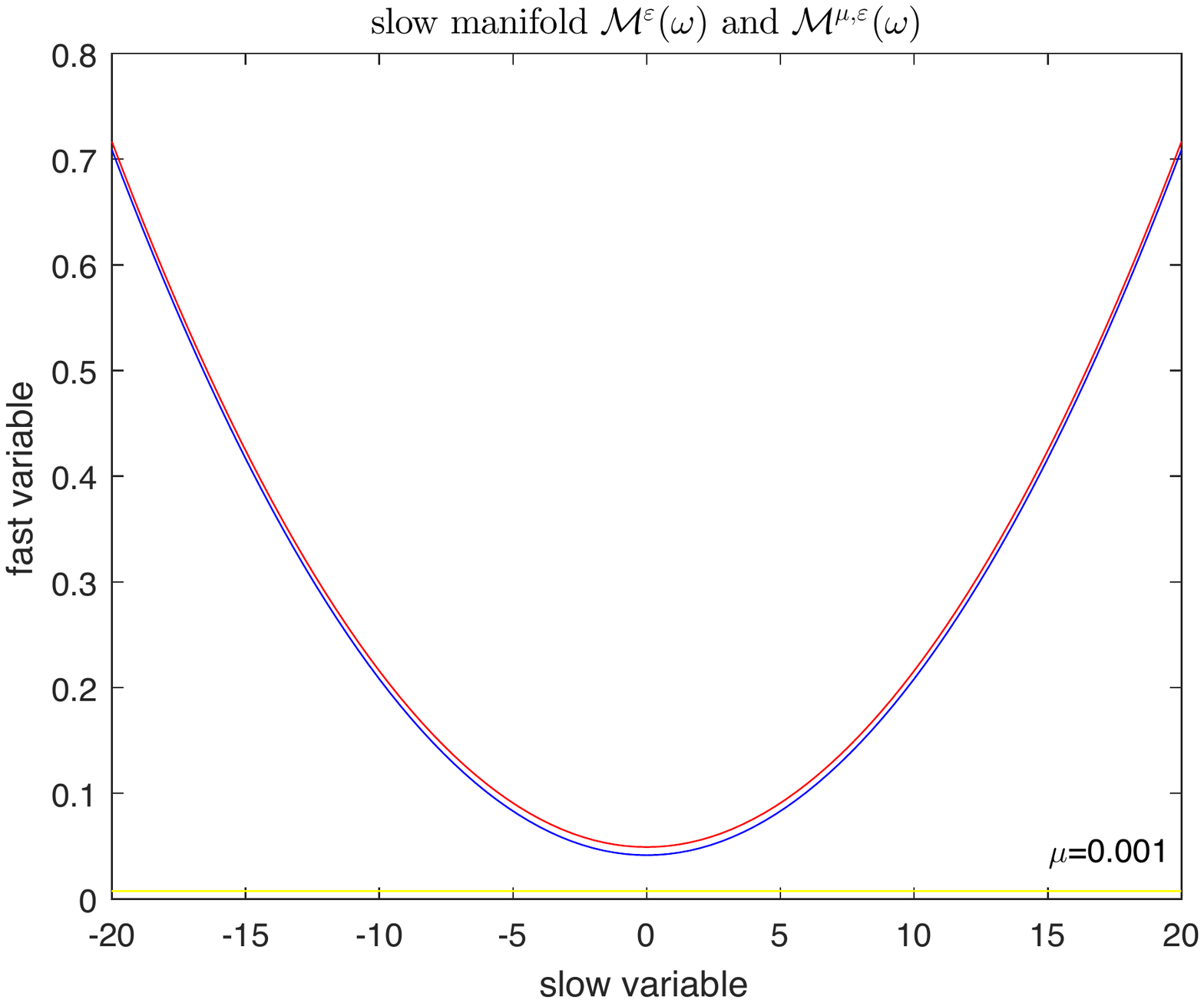}
\caption{(Color online) The slow manifold (one fixed sample) of system (\ref{exorgSDE}) with red line. And the slow manifold (the same sample with system (\ref{exorgSDE})) of system (\ref{exnewSDE}) with blue line. The Wong-Zakai approximation parameter $\mu=0.1$, $0.01$ and $0.001$ in the left, center and right of the figure respectively. The error is plotted with yellow line.}\label{figsm}
\end{figure}
Figure \ref{figsm2} shows that the evolvement of one sample of the random slow manifold will be affected by the noise through the Wiener shift and oscillate like the solution of stochastic dynamical system. It also shows that the graph of slow manifold of system ($\ref{exnewSDE}$) is more smoother than that of system ($\ref{exorgSDE}$) as time evolves.
\begin{figure}[H]
\centering
\includegraphics[width=2in]{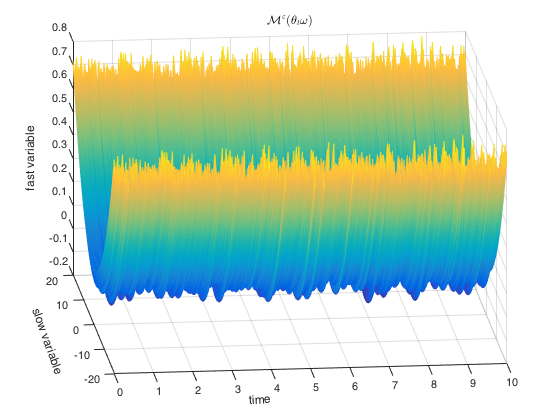}
\includegraphics[width=2in]{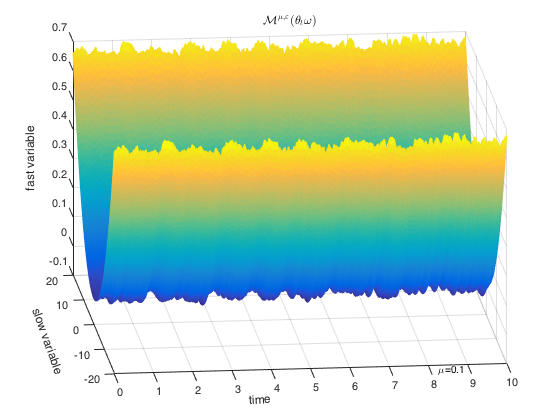}
\includegraphics[width=2in]{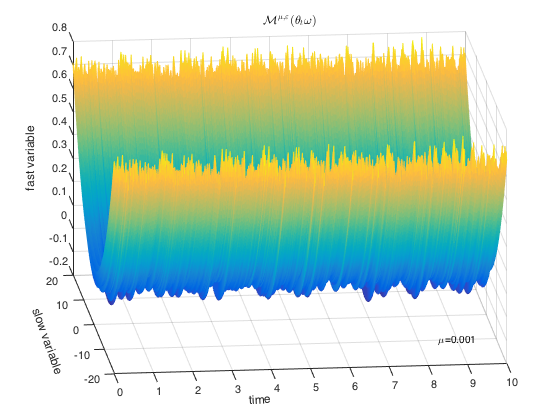}
\caption{(Color online) The random slow manifold of system (\ref{exorgSDE}) with one sample (the same used in figure \ref{figsm}) evolving under the Wiener shift in the first picture. The random slow manifold of system (\ref{exnewSDE}) with the same sample evolving under the Wiener shift with the Wong-Zakai approximation parameter $\mu=0.1$ and $0.001$ in the second and the third pictures respectively.}\label{figsm2}
\end{figure}

In order to show the exponential tracking property between the two different systems, we plot the graph of the random slow manifold of system ($\ref{exorgSDE}$) and ($\ref{exnewSDE}$) as time evolves and one solution of the original system ($\ref{exorgSDE}$) in Figure \ref{figexptrack}. 
\begin{figure}[H]
\centering
\includegraphics[width=2.5in]{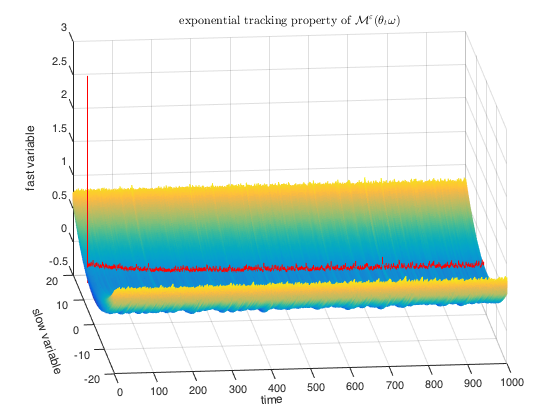}
\includegraphics[width=2.5in]{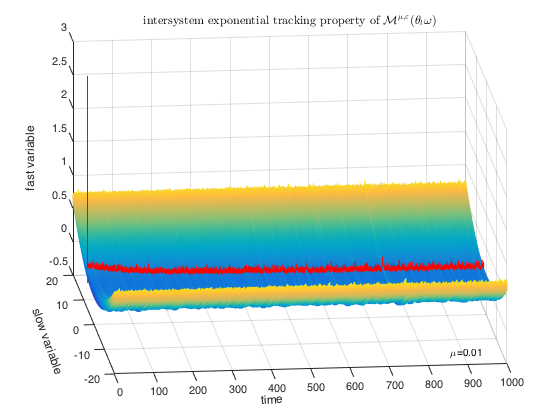}
\begin{center}(a)\hspace{8cm}(b)\end{center}
\caption{(Color online) Intersystem exponential tracking property when $\mu=0.01$, $\varepsilon=0.01$. One solution of the original system ($\ref{exorgSDE}$) with red line in both (a) and (b). The original random slow manifold of system ($\ref{exorgSDE}$) in (a) and the Wong-Zakai random slow manifold of system ($\ref{exnewSDE}$) in (b) as time evolves respectively.} \label{figexptrack}
\end{figure}
At last, we derive the reduced system of the original system ($\ref{exorgSDE}$) using the expansions to order o($\varepsilon^{2}$) of the original random slow manifold ($\ref{exsmf1}$) and Wong-Zakai random slow manifold ($\ref{exsmf2}$), respectively.
The reduction system of system ($\ref{exorgSDE}$) using the original random slow manifold ($\ref{exsmf1}$) of system ($\ref{exorgSDE}$) is
\begin{eqnarray}\label{ex:reduce orig}
\dot{v_{t}}^{\varepsilon}=0.001*v_{t}^{\varepsilon}-a*h^{\varepsilon}(\theta_{t}\omega,v^{\varepsilon}_{t})*v_{t}^{\varepsilon},
\end{eqnarray}
where
\begin{eqnarray*}
\begin{aligned}
h^{\varepsilon}(\theta_{t}\omega,v^{\varepsilon}_{t})
=&\frac{(v^{\varepsilon}_{t})^{2}}{600}+\varepsilon \bigg[-\frac{0.001(v^{\varepsilon}_{t})^{2}}{300}+\frac{0.1(v^{\varepsilon}_{t})^{4}}{180000}-\frac{0.01(v^{\varepsilon}_{t})^{2}}{300}\int_{-\infty}^{t}(s-t)e^{s-t}\mathrm{d}B_{s}(\omega)\bigg]+o(\varepsilon^{2})\\
&+\frac{\sigma}{\sqrt{\varepsilon}}\int_{-\infty}^{t}e^{\frac{s-t}{\varepsilon}}\mathrm{d}B_{s}(\omega).\\
\end{aligned}
\end{eqnarray*}

The reduction system of system ($\ref{exorgSDE}$) using Wong-Zakai random slow manifold ($\ref{exsmf2}$) of system ($\ref{exnewSDE}$) is
\begin{eqnarray}\label{ex:reduce wongzakai}
\dot{v_{t}}^{\varepsilon}=0.001*v_{t}^{\varepsilon}-a*h^{\mu,\varepsilon}(\theta_{t}\omega,v^{\varepsilon}_{t})*v_{t}^{\varepsilon},
\end{eqnarray}
where
\begin{eqnarray*}
\begin{aligned}
h^{\mu,\varepsilon}(\theta_{t}\omega,v^{\varepsilon}_{t})
=&\frac{(v^{\varepsilon}_{t})^{2}}{600}+\varepsilon \bigg[-\frac{0.001(v^{\varepsilon}_{t})^{2}}{300}+\frac{0.1(v^{\varepsilon}_{t})^{4}}{180000}-\frac{0.01(v^{\varepsilon}_{t})^{2}}{300}\int_{-\infty}^{t}(s-t)e^{s-t}z_{s}^{\mu}\mathrm{d}s\bigg]+o(\varepsilon^{2})\\
&+\frac{\sigma}{\sqrt{\varepsilon}}\int_{-\infty}^{t}e^{\frac{s-t}{\varepsilon}}z_{s}^{\mu}\mathrm{d}s(\omega).
\end{aligned}
\end{eqnarray*}
The estimation to system parameter $a$ using the  reduction system (\ref{ex:reduce orig}) based on the original random slow manifold is shown at the left side in figure \ref{estimation}. The estimation to system parameter $a$ using the  reduction system (\ref{ex:reduce wongzakai}) based on Wong-Zakai random slow manifold is shown at the right side in figure \ref{estimation}.

\begin{figure}[H]
\centering
\includegraphics[width=2.5in]{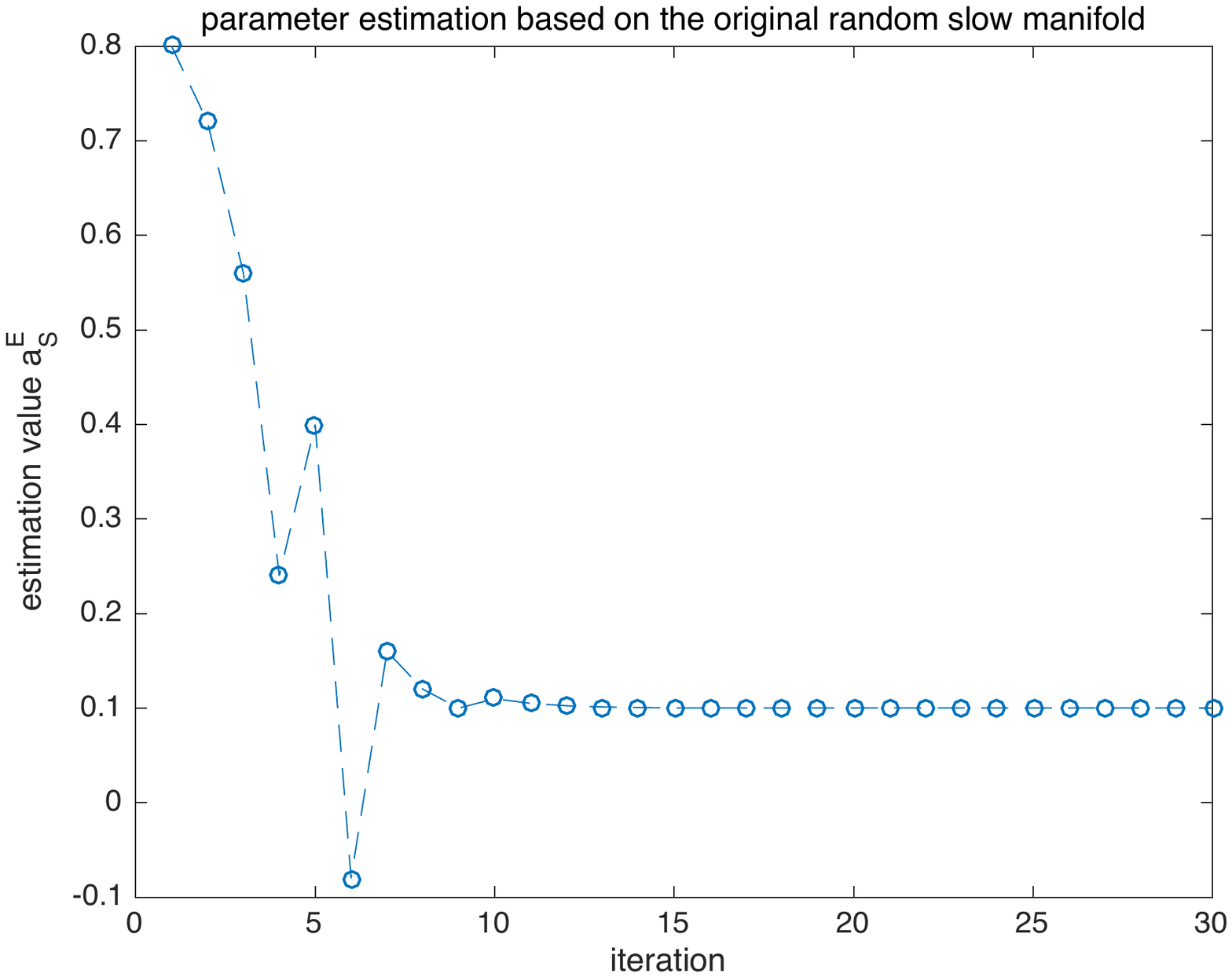}
\includegraphics[width=2.5in]{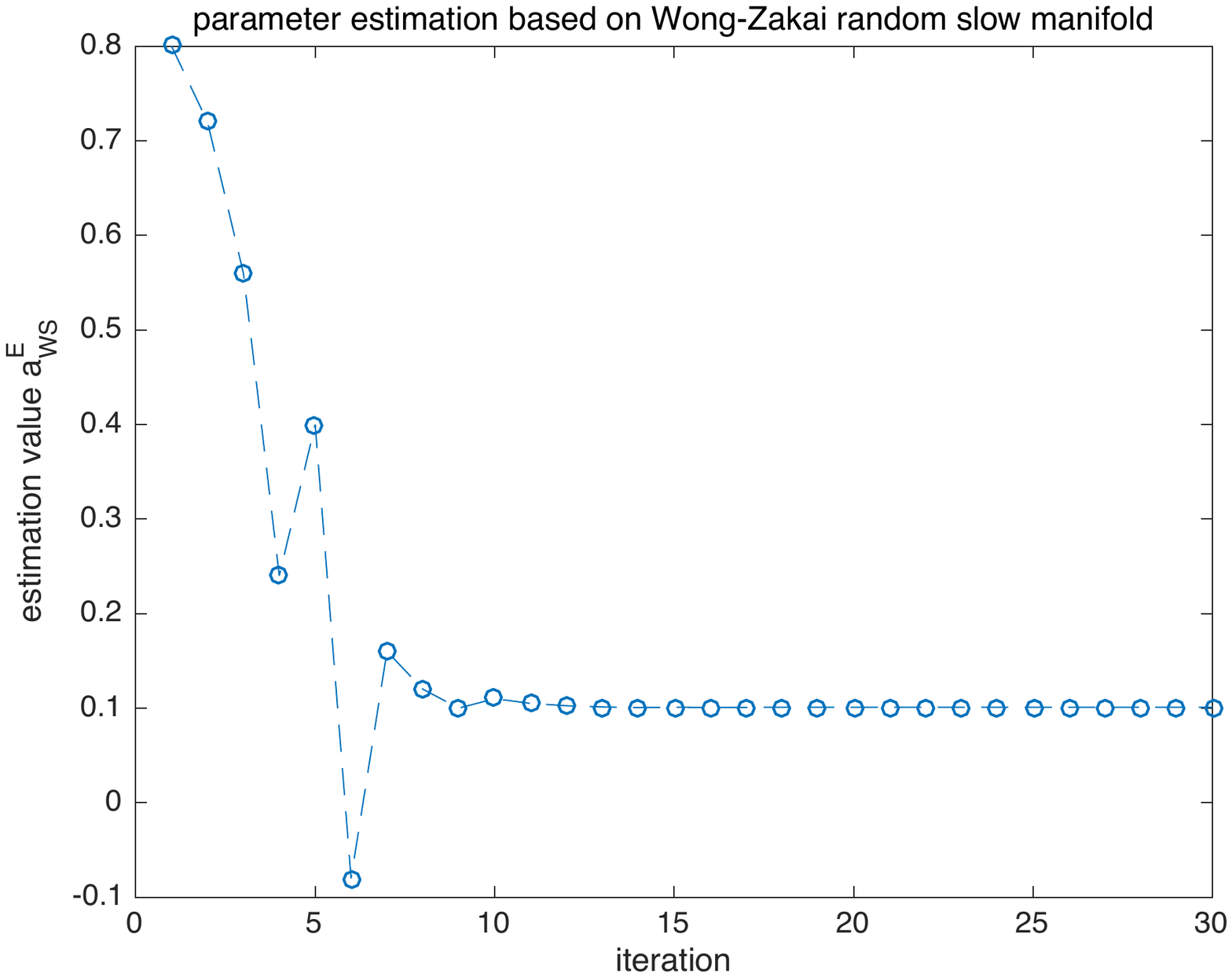}
\caption{The left side is the parameter estimation using system (\ref{ex:reduce orig})  with $\varepsilon=0.1$ and $a=0.1$. After $30$ iterations and $347.463716$ seconds, the estimation value $a_{S}^{E}$ is $0.1002$ with error $1.7547\times10^{-4}$, and the value of objective function is $0.0011$.\\
The right side is the parameter estimation using system (\ref{ex:reduce wongzakai})  with $\varepsilon=0.1$, $\mu=0.01$ and $a=0.1$. After $30$ iterations and $273.192621$ seconds, the estimation value $a_{WS}^{E}$ is $0.1009$ with error $9.3110\times 10^{-4}$, and the value of objective function is $0.0027$.} \label{estimation}
\end{figure}
From figure \ref{estimation}, the error of parameter estimation based on Wong-Zakai random slow manifold is nearly equal to that based on the original random slow manifold. While the simulation of parameter estimation based on Wong-Zakai random slow manifold needs less time than that based on the original random slow manifold. This indicates that the parameter estimation method based on Wong-Zakai random slow manifold is a good approximation to the method based on the original random slow manifold. Hence it is a good approximation method to the method using the original system directly according to \cite{Ren jian}.

Our parameter estimation method based on Wong-Zakai random slow manifold has three benefits.  First, it decrease the amount of  observational information, compared to the method using the original system. The former method only needs to know the value of slow variables whereas the latter method   needs to observe both the fast   and   slow variables.
The second benefit is that it reduces the cost for computation,  compared to the method based on the original random slow manifold. This can be seen from the reduction of simulation time. The third benefit is due to the simplification in numerical  simulation. It is easier to simulate RDEs than SDEs. A  RDE can simulated by deterministic Nelder-Mead simulation method (sample-wisely), but an SDE must use stochastic Nelder-Mead simulation method which is more complicated.

The accuracy of this estimation method also reflects the validity of our Wong-Zakai approximation for random slow manifold to  some extent.

\bigskip
{\bf Acknowledgements}
The authors thank Xiujun Cheng, Jian Ren and Xiaoli Chen for useful discussions about the program for parameter estimation, and  Hua Zhang for pointing out the theoretical basis about the computation of stochastic integration.

\end{document}